\documentclass[preprint]{imsart}

\usepackage{epsfig,graphics}
\usepackage{amssymb, amsmath,amsthm,array}
\usepackage[numbers]{natbib}
\usepackage{bm}
\usepackage{color}
\usepackage{lipsum}
\usepackage{appendix}
\renewcommand{\theequation} {\arabic{section}.\arabic{equation}}
\usepackage{hyperref}

\def \V {{\rm V}}
\def \min {{\rm min}}
\def \P {\mathbf{P}_\theta}
\def \R {\mathbb{R}}

\def \Xs {{\mathbb X}_{1:s}}
\def \Ls {{\mathbb L}_{1:s}}

\newcommand{\ve}{\varepsilon}

\newtheorem{lemma}{\sc Lemma}

\newtheorem{theorem}{\sc Theorem}

\newtheorem{proposition}{\sc Proposition}

\newtheorem{corollary}{\sc Corollary}

\renewcommand{\abstractname}{\sc Abstract}

\newcommand{\simo}[1]{{\color{red}#1}}

\def\qed{\space$\square$ \par \vspace{.15in}}

\begin{document}

\begin{frontmatter}
\title{Variable selection, monotone likelihood ratio and group sparsity}
%\title{Variable selection under monotone likelihood ratio with application to group sparse model}
\runtitle{ Variable selection, MLR and group sparsity}
%\thankstext{T1}{Footnote to the title with the ``thankstext'' command.}

%\begin{aug}
%%%%%%%%%%%%%%%%%%%%%%%%%%%%%%%%%%%%%%%%%%%%%%%
%% Only one address is permitted per author. %%
%% Only division, organization and e-mail is %%
%% included in the address.                  %%
%% Additional information can be included in %%
%% the Acknowledgments section if necessary. %%
%%%%%%%%%%%%%%%%%%%%%%%%%%%%%%%%%%%%%%%%%%%%%%%
%\author[A]{\fnms{???} \snm{???}\ead[label=e1]{???@???}},
%\author[B]{\fnms{???} %%\snm{???}\ead[label=e2,mark]{???@???}}
%\and
%\author[B]{\fnms{???} %\snm{???}\ead[label=e3,mark]{???@???}}
%%%%%%%%%%%%%%%%%%%%%%%%%%%%%%%%%%%%%%%%%%%%%%
%% Addresses                                %%
%%%%%%%%%%%%%%%%%%%%%%%%%%%%%%%%%%%%%%%%%%%%%%
%\address[A]{???, \printead{e1}}

%%\address[B]{???, \printead{e2,e3}}
%\end{aug}

\begin{aug}
\author[A]{\fnms{Cristina} \snm{Butucea}\ead[label=e1]{cristina.butucea@ensae.fr}},
\author[B]{\fnms{Enno} \snm{Mammen}
\ead[label=e2]{mammen@math.uni-heidelberg.de}},
\author[C]{\fnms{Mohamed} \snm{Ndaoud}\ead[label=e3]{ndaoud@essec.edu}}
\and
\author[A]{\fnms{Alexandre B.} \snm{Tsybakov}\ead[label=e4]{alexandre.tsybakov@ensae.fr}}
%\thankstext{t1}{Some comment}
%\thankstext{t2}{First supporter of the project}
%\thankstext{t3}{Second supporter of the project}
%\runauthor{F. Author et al.}

\address[A]{CREST, ENSAE, IP Paris\\
5, ave. Henry Le Chatelier \\
91120 Palaiseau, France \\
\printead{e1,e4}}

\address[B]{Institute for Applied Mathematics\\
Heidelberg University\\
69120 Heidelberg, Germany\\
\printead{e2}}

\address[C]{Department of Information Systems, Decision Sciences and Statistics\\
ESSEC Business School\\
95000 Cergy, France\\
\printead{e3}}
%\phantom{E-mail:\ }\printead*{e2}}

\end{aug}

\begin{abstract}
In the pivotal variable selection problem, we derive the exact non-asymptotic minimax selector over the class of all $s$-sparse vectors, which is also the Bayes selector with respect to the uniform prior. While this optimal selector is, in general, not realizable in polynomial time, we show that its tractable counterpart (the scan selector)  attains the minimax expected Hamming risk to within factor 2, and is also exact minimax with respect to the probability of wrong recovery. As a consequence, we establish explicit lower bounds under the monotone likelihood ratio property  and we obtain a tight characterization of the minimax risk in terms of the best separable selector risk.  
We apply these general results to derive necessary and sufficient conditions of exact and almost full recovery in the location model with light tail distributions and in the problem of group variable selection under Gaussian noise.

%We study the variable selection problem under Hamming loss when the underlying distributions belong to a family with monotone likelihood ratio. We establish general non-asymptotic lower bounds that improve on \cite{butucea2018variable} and allow us to deduce non-asymptotic necessary and sufficient conditions for exact and almost full variable selection. A tight characterization of the risk is given in the regime where exact recovery is not possible. These results are applied to models with log-concave distributions, to group variable selection in the Gaussian vector model and its generalization to groups of sub-Gaussian vectors with given covariance matrix. We also provide numerical results supporting  the theory.
\end{abstract}



\end{frontmatter}

%\noindent {\it AMS 2000 subject classifications}: 62G07; 62G20

%\newpage

\section{Introduction}\label{intro}
\setcounter{equation}{0}
Assume that we observe independent random variables $X_{1},\dots,X_{d}$ on a measurable space $(\mathcal{X},\mathcal{U})$ such that $s$ among them are distributed according to the probability measure $P_{1}$ and the others are distributed according to the probability measure $P_{0}$. We assume that {$d> 2$}, $1\le s<d$ and $P_{0}\neq P_{1}$. Let $f_{0}$ and $f_{1}$ be densities of $P_{0}$ and $P_{1}$ with respect to some dominating measure that we will further denote by $\mu$. Denote by $\eta=(\eta_{1},\dots,\eta_{d})$ the vector such that $\eta_{j}=1$ if the distribution of $X_{j}$ is $P_{1}$ and $\eta_{j}=0$ if it is $P_{0}$. Define $\Theta_d(s)$ as the set of all vectors $\eta \in \{0,1\}^{d}$ with exactly $s$ non-zero components. The components $j$ corresponding to $\eta_j=1$ can be interpreted as relevant variables. We state the problem of variable selection  as the problem of estimating the binary vector
$\eta$. 

As estimators (possibly randomized) of $\eta$, we consider any measurable functions $\widehat \eta=\widehat \eta(X_1,\dots,X_d,\zeta)$ of $(X_1,\dots,X_d,\zeta)$  taking values in $\{0,1\}^d$, where $\zeta$ is a random variable with values in a measurable space $(\mathcal{Z},\mathcal{V})$ independent of $(X_1,\dots,X_d)$. 
Such estimators will be called {\it selectors}.
We define the loss of a selector $\widehat \eta$  by the Hamming distance between $\widehat \eta$ and $\eta$, that is, by the number of positions at which $\widehat \eta$ and $\eta$ differ:
$$
|\widehat \eta-\eta|:=\sum_{j=1}^d |\widehat \eta_j-\eta_j|= \sum_{j=1}^d \mathbf{1}(\widehat \eta_j\ne \eta_j).
$$
Here $\widehat \eta_j$ and $ \eta_j$ are the $j$th components of $\widehat \eta$ and $ \eta$, respectively, and $\mathbf{1}(\cdot)$ denotes the indicator function. 
The performance  of a selector $\widehat \eta$ is measured by its expected Hamming risk ${\mathbf E}_\eta |\widehat \eta - \eta|$.  Here ${\mathbf E}_\eta$ denotes the expectation with respect to probability measure ${\mathbf P}_\eta$ of $(X_1,\dots,X_d,\zeta)$ for given $\eta$.
As a benchmark for the class $\Theta_d(s)$ we consider the corresponding minimax risk of variable selection
\begin{align*}%\label{minimaxrisk}
\inf_{\tilde \eta} \sup_{\eta \in \Theta_d(s)} \, {\mathbf E}_\eta |\tilde \eta - \eta|,
\end{align*}
where $\inf_{\tilde \eta}$ denotes the infimum over all selectors.
In this paper, we find optimal selectors achieving the minimax risk. We also study such asymptotic properties as exact recovery and almost full recovery. %We will use our non-asymptotic risk bounds to study the following asymptotic concepts. 
For a sequence $(s_d)_{d \geq 1}$ we
 say that a selector $\widehat \eta$ achieves exact recovery for $\Theta_d(s_d)$  if
\begin{align*}
\limsup_{d\to \infty} \sup_{\eta \in \Theta_d(s_d)} \, {\mathbf E}_\eta |\widehat \eta - \eta| =0
\end{align*}
 and it achieves almost full recovery for $\Theta_d(s_d)$  if
\begin{align*}
\limsup_{{d}\to \infty} \sup_{\eta \in \Theta_d(s_d)} \, \frac{1}{s_d} {\mathbf E}_\eta |\widehat \eta - \eta| =0.
\end{align*}
Along with the Hamming risk, another popular risk measure for variable selection is the probability of wrong recovery ${\mathbf P}_\eta( S_{\widehat\eta} \ne S_{\eta})$, where $S_\eta= \{ j: \,  \eta_j\ne 0\}$. 
We have
\begin{equation} \label{eq:risks}
{\mathbf P}_\eta(S_{\widehat\eta} \ne S_{\eta}) = {\mathbf P}_\eta(|\widehat \eta - \eta|\ge 1) \le {\mathbf E}_\eta |\widehat \eta - \eta|.
\end{equation}
Thus, ${\mathbf P}_\eta(S_{\widehat\eta} \ne S_{\eta})$ is, in fact, the Hamming risk with indicator loss. Bounding from above the expected Hamming risk provides a stronger result than bounding the probability of wrong recovery. We also note that 
instead of the $\ell_0$-sphere $\Theta_d(s)$ in the above definitions one can consider the $\ell_0$-ball $\{\eta: |\eta|_0\le s\}$ as it was done in \cite{butucea2018variable}. Here,  $|\eta|_0$ denotes the number of non-zero components of $\eta$.

The problem of 
variable selection as defined above was introduced in \cite{butucea2018variable}. In what follows, we will call it the {\it pivotal selection problem}. It was shown in \cite{butucea2018variable} that it gives a key to solve other problems of variable selection, where instead of two measures $P_0$ and $P_1$ one has families of measures. 
We also mention a related prior work dealing with asymptotic settings, cf. \cite{IS2014} and \cite{BS2017} for sparse additive models with smooth signals. 
One of the contributions in \cite{butucea2018variable} was to establish non-asymptotic bounds for the minimax Hamming risk. Based on these bounds and exploiting the monotone likelihood ratio property, \cite{butucea2018variable} derived sharp conditions of exact and almost full recovery under the sparse Gaussian shift model.
 Similar ideas were further developed for linear regression \cite{ndaoud2019optimal,ndaoud2020scaled,gao2019iterative,gao2019fundamental,zadik,zadik2019}, multiple testing \cite{roquain2020false,abraham2021sharp,arias2017distribution}, sparse confidence sets \cite{belitser2021uncertainty,ning2020sparse} and 
$k$-ranking  \cite{chen2020partial}. 

Prior to these developments, a lot of work was devoted to variable selection in high-dimensional linear regression under sparsity.
In particular, the Lasso, Dantzig selector, other penalized methods as well as marginal regression were analyzed in detail; see, for example, \cite{MS,ZH,Lounici2008,fletcher,wainwrightaLasso,Zhang2010,rad,saligrama1,caiOMP,genovese-jin-wasserman,jin1} and the references therein. These papers focus on recovery of the sparsity pattern $S_\theta$ of the unknown regression parameter $\theta\in \R^d$. They propose selectors such that the probability of wrong recovery ${\mathbf P}_\eta(S_{\widehat\eta}  \ne S_\theta)$ is close to 0 in an asymptotic sense. 
Sparse group variable selection is considered from the same perspective in \cite{lounici2010oracle} based on the group Lasso estimator. Group selection was recently analyzed under a different angle in \cite{reeve2021optimal}, where one can find further references on the topic.   
Under a Bayesian setting and the assumption that $s\sim d^{\beta}$ for some $\beta\in(0,1)$, a refined asymptotic analysis of variable selection for Gaussian linear regression  is provided in  \cite{genovese-jin-wasserman,jin1}.

A great deal of attention in the literature on variable selection was devoted to models with Monotone Likelihood Ratio (MLR) property. We state it as follows. 

\smallskip

\noindent\textbf{MLR property.}
%\noindent
{\it The observations $X_i$ take values in a subset of $\R$, the densities $f_0$ and $f_1$ have the same support, and $\frac{f_{1}}{f_{0}}(\cdot)$ is an increasing function on the support of the densities.}

\smallskip

This property covers a large class of exponential family distributions. It is satisfied for location models with log-concave density. It is also satisfied for the chi-square location model where $f_{0}$ and $f_{1}$ are the standard chi-square and a non-central chi-square densities, respectively. Such a model arises in connection with group variable selection, see Section \ref{sec:grouped} below.  

The present paper is organized  as follows.  
First, in the pivotal selection problem, we 
find the exact minimax selector (in non-asymptotic sense) over
the class of all $s$-sparse vectors $\Theta_d(s)$, which is also the Bayes selector with respect to the uniform prior. While this minimax selector is, in general, not realizable in polynomial time, we show that its tractable counterpart (the scan selector)  attains the minimax Hamming risk to within factor 2, and is also exact minimax with respect to the probability of wrong recovery. Based on these results, we derive explicit lower bounds under the MLR property  and obtain a tight characterization of the minimax risk in terms of the best separable selector risk.  We propose two methods of deriving accurate lower bounds.  One of them is based on using a block prior that draws one entry at random from each of $s$ blocks of size $d/s$, while the other one exploits uniform prior on binary vectors with support size $s$ and the exact minimax and Bayes selector.

Second, we apply these general results in two specific settings of interest. In Section~\ref{sec:log-concave}, we apply them to the location model with light tail distributions. In Section~\ref{sec:grouped}, we consider the problem of group variable selection under Gaussian noise. %as well as  its extension  to sub-Gaussian observation scheme with general covariance matrix. 
For both problems, we derive necessary and sufficient conditions of exact and almost full recovery based on the general results of Section~\ref{sec:minimax-and-bayes}. 
%Finally, Section \ref{sec:plot} provides a numerical experiment supporting the theory.

\section{Minimax and Bayes solutions for the pivotal selection problem}\label{sec:minimax-and-bayes}
\setcounter{equation}{0}

%Recall that we consider the setting where the MLR property is satisfied, that is, $X_i$'s are real-valued, the densities $f_0$ and $f_1$ have the same support and $\frac{f_{1}}{f_{0}}(\cdot)$ is an increasing function. 
We first recall the bounds \cite{butucea2018variable} on the minimax Hamming risk in the pivotal selection problem. The main purpose in \cite{butucea2018variable} was to show that the minimax risk is closely related to the risk of the best separable selector, that is, to the quantity
\begin{align}\label{eq:old_bound}
\Psi_{\rm sep}(d,s)&=s{P}_1
(sf_1(X_1) < (d-s)f_0(X_1))
\\ \nonumber
&\quad
+
(d-s)
{P}_0
(sf_1(X_1) \ge (d-s)f_0(X_1)).
\end{align}
By definition, a selector $\tilde\eta$ is called {\it separable} if its $j$th component $\tilde\eta_j$ depends only on $X_j$ for $j=1,\dots,d$. Using the Neyman-Pearson lemma, it is not hard to check that the selector with components \cite{butucea2018variable}
\begin{equation}\label{eq:separable}
\widehat\eta_j=\mathbf{1}\{sf_1(X_j) \ge (d-s)f_0(X_j)\}
\end{equation}
minimizes the Hamming risk among all separable selectors for any $\eta$ such that $|\eta|_{0}=s$. The value $\Psi_{\rm sep}$ defined in \eqref{eq:old_bound} gives the Hamming risk of this best separable selector for all $\eta$ such that $|\eta|_{0}=s$. Thus, we have
\begin{equation}\label{eq:minimax-separable}
 \Psi_{\rm sep}(d,s) =   \underset{{\rm separable}\, \tilde{\eta}}{\inf} \ \underset{|\eta|_{0}= s}{\sup} \mathbf{E}_{\eta} |\tilde{\eta} - \eta|.
\end{equation} 
Moreover, $\Psi_{\rm sep}$ is related to the minimax Hamming risk over all selectors as stated in the following theorem.

\begin{theorem}\label{th:BNST18}[ \cite{butucea2018variable}, Theorem 3.2]
For any $s' \in (0,s]$,
\begin{equation}\label{eq:BNST18:1}
\underset{\tilde{\eta}}{\inf} \underset{|\eta|_{0}\leq s}{\sup} \mathbf{E}_{\eta} |\tilde{\eta} - \eta|\geq \frac{s'}{s}\Psi_{\rm sep}(d,s) - 4s'\exp\left(-\frac{(s-s')^{2}}{2s}\right).
\end{equation}
Furthermore, the selector $\widehat\eta$ with components \eqref{eq:separable} is such that
\begin{equation}\label{eq:BNST18:2}
 \underset{|\eta|_{0}\leq s}{\sup} \mathbf{E}_{\eta} |\widehat{\eta} - \eta|\leq \frac{d}{d-s}\Psi_{\rm sep}(d,s).
\end{equation}
\end{theorem}
 It is shown in \cite{butucea2018variable} that bounds  \eqref{eq:BNST18:1} (with a suitable choice of $s'$) and \eqref{eq:BNST18:2} are accurate enough to give asymptotically sharp conditions of exact and almost full recovery for a selection problem in the Gaussian shift model. However, the behavior of the minimax Hamming risk  with non-Gaussian distributions is not understood. In general, finding $s'$ maximizing the bound in \eqref{eq:BNST18:1} is a complicated problem, and it is not granted that the solution would fit the upper bound, even asymptotically. Note also that the lower bound \eqref{eq:BNST18:1} is not valid when the maximum is taken over the sphere  $\{|\eta|_{0}= s\}$ rather than over the ball $\{|\eta|_{0}\leq s\}$ (the proof in \cite{butucea2018variable}  uses in a crucial way the fact that the set of underlying binary sequences is 
 $\{|\eta|_{0}\leq s\}$). %It seems counter-intuitive as the number of elements in both sets is of the same order of magnitude.

In this section, we develop the analysis of the pivotal selection problem free of these drawbacks. First, we improve upon \cite{butucea2018variable} by finding exact non-asymptotic minimax selector over
the class of all $s$-sparse vectors $\{|\eta|_{0}= s\}$, which is also the Bayes selector with respect to the uniform prior. Noteworthy, it is not a separable selector and, in general, it is not realizable in polynomial time. We provide tight characterizations of its risk in terms of the risk of its tractable counterpart (the scan selector), and of the best separable selector risk $\Psi_{\rm sep}$ under the MLR assumption. 

\subsection{Solution for $s = 1$ and block prior lower bound}
\label{sec:block}

In order to get more intuition, we start with the case $s=1$, so that $|\eta|_{0}=1$.  
In this case, we observe $X_{1},\dots,X_{d}$, where exactly one of $X_{i}$'s has density $f_{1}$ while all other have density $f_{0}$. 
Let $\pi_1$ be the uniform prior distribution on all $\eta$ with $|\eta|_{0}=1$, i.e. the distribution of $\eta=\eta({\sf k})$ with components
$$ \eta_{i}({\sf k}) = \mathbf{1}({\sf k}=i), \  \ i=1,\dots,d,$$
where $\sf k$ is a random variable uniformly distributed on $\{1,\dots,d\}$. Consider the selector $\eta^{*}$ with components
\begin{equation}\label{eq:lower_bound_01}
\eta^{*}_{j} = \mathbf{1}\left\{ f_{1}(X_{j})\prod_{i \neq j}f_{0}(X_{i}) \geq  \sum_{k \neq j}\left(f_{1}(X_{k})\prod_{i \neq k}f_{0}(X_{i})\right) \right\}
, \ \ 
\end{equation}
for $j=1,\dots,d$.
\begin{theorem}\label{thm:lower:sparse:one}
Selector defined by \eqref{eq:lower_bound_01} is minimax on the set $\{|\eta|_{0} = 1\}$
and Bayes with respect to the prior $\pi_1$:
\begin{align*}
\underset{\tilde{\eta}}{\inf} \underset{|\eta|_{0} = 1}{\sup} \mathbf{E}_{\eta} |\tilde{\eta} - \eta| 
&= \underset{|\eta|_{0} = 1}{\sup}\mathbf{E}_{\eta} |\eta^{*} - \eta|
= \mathbf{E}_{e_1} |\eta^{*} - e_1|,
\\
\underset{\tilde{\eta}}{\inf} \ \mathbb{E}_{\pi_1} \mathbf{E}_{\eta} |\tilde{\eta} - \eta| 
&= \mathbb{E}_{\pi_1}\mathbf{E}_{\eta} |\eta^{*} - \eta|
= \mathbf{E}_{e_1} |\eta^{*} - e_1|,
\end{align*}
where $\underset{\tilde{\eta}}{\inf}$ is the infimum over all selectors, $ \mathbb{E}_{\pi_1}$ denotes the expectation over $\eta$ with distribution $\pi_1$, and $e_1=(1,0,\dots,0)$ is the first canonical basis vector
of $\mathbf{R}^d$. 
\end{theorem}
\begin{proof}
Bounding the minimax risk by the Bayes risk we have
\begin{equation}\label{eq:lower_bound_0}
\underset{\tilde{\eta}}{\inf} \underset{|\eta|_{0} = 1}{\sup} \mathbf{E}_{\eta} |\tilde{\eta} - \eta| \geq \underset{\tilde{\eta}}{\inf} \ \mathbb{E}_{\pi_1}\mathbf{E}_{\eta} |\tilde{\eta} - \eta|.
\end{equation}
In the rest of this proof, we fix the randomization $\zeta$  and we consider $\underset{\tilde{\eta}}{\inf}$ as infimum over non-randomized selectors $\tilde\eta (X_1,\dots,X_d)$. The obtained lower bound for the Bayes risk obviously remains true with the infimum over randomized selectors as this bound does not depend on $\zeta$, and $\zeta$ is independent of $(X_1,\dots,X_d)$.

Moving in \eqref{eq:lower_bound_0} the infimum under the sum and using the definition of $\pi_1$ we obtain
$$
 \underset{\tilde{\eta}}{\inf}\, \mathbb{E}_{\pi_1}\mathbf{E}_{\eta} |\tilde{\eta} - \eta| \geq  \sum_{j=1}^{d} \underset{\tilde{\eta}_{j}}{\inf}\, {\rm E}_{\sf k}\mathbf{E}_{\eta({\sf k})}|\tilde{\eta}_{j} - \mathbf{1}({\sf k}=j)|,
$$
where ${\rm E}_{\sf k }$ denotes the expectation with respect to random variable ${\sf k}$.
We have
\begin{equation}\label{eq:lower_bound_1}
{\rm E}_{\sf k }\mathbf{E}_{\eta({\sf k})}|\tilde{\eta}_{j} - \mathbf{1}({\sf k}=j)| = \frac{1}{d}\left(\mathbf{E}_{e_j}(1-\tilde{\eta}_{j}) + \sum_{k \neq j}\mathbf{E}_{e_k}(\tilde{\eta}_{j})\right),
\end{equation}
where $e_k$ is the $k$th canonical basis vector of $\mathbf{R}^d$, so that $\mathbf{E}_{e_k}$ is the expectation with respect to the density $f_{1}(X_{k})\prod_{i \neq k}f_{0}(X_{i})$.
Arguing as in the Neyman-Pearson lemma, we conclude that $ \eta_{j}^{*}$ 
minimizes \eqref{eq:lower_bound_1} among all $\{0,1\}$-valued statistics, and hence
\begin{equation}\label{eq:lower_bound_02}
\underset{\tilde{\eta}}{\inf}\, \mathbb{E}_{\pi_1}\mathbf{E}_{\eta} |\tilde{\eta} - \eta| \geq  \sum_{j=1}^{d}  {\rm E}_{\sf k }\mathbf{E}_{\eta({\sf k})}|\eta^{*}_{j} - \mathbf{1}({\sf k}=j)| = \mathbb{E}_{\pi_1}\mathbf{E}_{\eta} |\eta^{*} - \eta|.
\end{equation}
Next, notice that, by permutation invariance, the distribution of $|\eta^*-\eta|$ under ${\mathbf P}_\eta$ is the same for all $\eta$ such that $|\eta|_{0} = 1$. Thus,
\begin{equation}\label{eq:lower_bound_3}
\mathbb{E}_{\pi_1}\mathbf{E}_{\eta} |\eta^{*} - \eta|
= \mathbf{E}_{e_1} |\eta^{*} - e_1| = \underset{|\eta|_{0} = 1}{\sup} \mathbf{E}_{\eta} |\eta^{*} - \eta|.
\end{equation}
The theorem now follows by combining 
\eqref{eq:lower_bound_02} and \eqref{eq:lower_bound_3}.
\end{proof}
%
%\begin{remark}
% Based on the fact that 
%$$
%\sum_{k\neq j}\frac{f_{1}}{f_{0}}(X_{k}) > \frac{f_{1}}{f_{0}}\left(\underset{k \neq j}{\max}X_{k}\right),
%$$
%and the MLR property, it is worth noticing that $|\eta^{*}|_{0} \leq 1 $. Hence the Bayesian selector $\eta^{*}$ selects at most one component (the maximum). The selection only happens provided that the maximum is large enough {(\color{red} what it means?)}. This result is stronger than the classical lower bound  {(\color{red} what is this bound?)} since we do not constraint the selectors to have sparsity exactly equal to $1$ nor to be separable.
%\end{remark}
%

 Theorem \ref{thm:lower:sparse:one} is valid for $X_i$'s of arbitrary nature and does not require the MLR property. Assume now that $X_i$'s are real-valued and the MLR property as stated
above holds. Note first that in this case the values $f_0(X_i)$ are $\mathbf{P}_\eta$-a.s. positive for any $\eta$, so that the  components of selector \eqref{eq:lower_bound_01} can be written as
 $$
 \eta_{j}^{*} = \mathbf{1}\left\{ \frac{f_{1}}{f_{0}}(X_{j}) \geq \sum_{k \neq j}\frac{f_{1}}{f_{0}}(X_{k})\right\}, \ \ j=1,\dots,d.
 $$
 On the other hand, $\mathbf{P}_\eta$-a.s.  for any $\eta$ we have:
 $$
 \sum_{k\neq j}\frac{f_{1}}{f_{0}}(X_{k}) > \frac{f_{1}}{f_{0}}\left(\underset{k \neq j}{\max}\,X_{k}\right).
 $$
 Using the MLR property we conclude that $\eta_{j}^{*}\le \mathbf{1}(X_{j} > \underset{k\ne j}{\max} \,X_{k})$ ($\mathbf{P}_\eta$-a.s. for any $\eta$ ). 
 Therefore, under the MLR property we have $|\eta^{*}|_{0} \leq 1 $ ($\mathbf{P}_\eta$-a.s. for any $\eta$ ), that is, the
minimax estimator  $\eta^{*}$ selects at most one  of $X_i$'s\footnote{Here and in what follows, we say that an estimator $\widehat\eta$ selects $X_i$ if its $i$th component $\widehat\eta_i$ is
equal to 1.}. Moreover, if one  of $X_i$'s is selected, it must be
the maximal order statistic of the sample. In the sequel, we often tacitly
assume (without mentioning it explicitly) that the considered relations between
random variables hold $\mathbf{P}_\eta$-a.s. for any $\eta$.

Under the MLR property, Theorem 2 implies the following lower
bound.
 
\begin{proposition}\label{prop:lower-bound-s=1}
Assume that 
the MLR property holds. Then
\begin{equation}\label{lower-bound-s=1}
\underset{\tilde{\eta}}{\inf} \underset{|\eta|_{0} = 1}{\sup} \mathbf{E}_{\eta} |\tilde{\eta} - \eta| \geq \mathbf{P}_{e_1}\left( X_{1} \leq \underset{j \geq 2}{\max} \,X_{j} \right).
\end{equation}
Furthermore,
$$
\underset{\tilde{\eta}}{\inf} \underset{|\eta|_{0} = 1}{\sup} \mathbf{E}_{\eta} |\tilde{\eta} - \eta| \geq (1-e^{-1})F_{1}\left( F_{0}^{-1}(1-1/(d-1))\right),
$$
where $F_{0}$ and $F_{1}$ are cumulative distribution functions corresponding to $f_{0}$ and $f_{1}$, respectively, and $F_{0}^{-1}(t)=\min\{v: F_0(v)=t\}$ for $t\in (0,1)$ is the inverse of $F_0$.
\end{proposition}
\begin{proof}
Using Theorem~\ref{thm:lower:sparse:one} we get
$$
\underset{\tilde{\eta}}{\inf} \underset{|\eta|_{0} = 1}{\sup} \mathbf{E}_{\eta} |\tilde{\eta} - \eta| \geq \frac{1}{d}\sum_{j=1}^{d}\sum_{k=1}^{d}\mathbf{E}_{e_k}|\eta^{*}_{j} - \mathbf{1}(k=j)|.
$$
It follows that
$$
\underset{\tilde{\eta}}{\inf} \underset{|\eta|_{0} = 1}{\sup} \mathbf{E}_{\eta} |\tilde{\eta} - \eta| \geq \frac{1}{d}\sum_{j=1}^{d}\mathbf{E}_{e_j}|\eta^{*}_{j} - 1| = \mathbf{P}_{e_1}\left( \frac{f_{1}}{f_{0}}(X_{1}) \leq \sum_{j \geq 2}  \frac{f_{1}}{f_{0}}(X_{j}) \right).
$$
Using the fact that 
$$
\sum_{j \geq 2}  \frac{f_{1}}{f_{0}}(X_{j}) > \frac{f_{1}}{f_{0}}\big(\underset{j \geq 2}{\max}\,X_{j}\big),
$$
and the MLR property we get \eqref{lower-bound-s=1}.

Next, set $\alpha= F_{0}^{-1}(1-1/(d-1))$. We have
$$
\mathbf{P}_{e_1}\left( X_{1} \leq \underset{j \geq 2}{\max}\, X_{j} \right) \geq \mathbf{P}_{e_1}\left( X_{1} \leq \alpha \right)\mathbf{P}_{e_1}\left( \alpha \leq \underset{j \geq 2}{\max} \,X_{j} \right) = F_1(\alpha) \mathbf{P}_{e_1}\left( \alpha \leq \underset{j \geq 2}{\max}\, X_{j} \right),
$$
where
$$
\mathbf{P}_{e_1}\left( \alpha \leq \underset{j \geq 2}{\max}\, X_{j} \right) = 1 - (F_{0}(\alpha))^{d-1} = 1 - (1-1/(d-1))^{d-1} \geq 1-e^{-1}.
$$
\end{proof}
Proposition \ref{prop:lower-bound-s=1} can be used to obtain a lower bound in the case of general
sparsity $s$ via a block sparsity argument. The strategy is to divide the data
into $s$ blocks and to apply Proposition \ref{prop:lower-bound-s=1} on each block separately. It leads to
the following lower bound.

\begin{proposition}\label{proposition:lower:sparse:1}
Let  $s < d/2$. Assume that 
the MLR property holds. Then
$$
\underset{\tilde{\eta}}{\inf} \underset{|\eta|_{0} = s}{\sup} \mathbf{E}_{\eta} |\tilde{\eta} - \eta| \geq (1-e^{-1})sF_{1}\left( F_{0}^{-1}(1-1/(\lfloor d/s\rfloor-1))\right).
$$
\end{proposition}
The proof of this proposition is given in the Appendix.
As a consequence, we derive the following corollary.
\begin{corollary}\label{cor:almost_rec}
Assume that 
the MLR property holds and  
$s < d/2$. If there exists $A>1$ such that $F^{-1}_{0}(1-1/(\lfloor d/s\rfloor-1)) > F^{-1}_{1}(1/A)$ 
then
$$
\underset{\tilde{\eta}}{\inf} \underset{|\eta|_{0} = s}{\sup} \frac1s \mathbf{E}_{\eta} |\tilde{\eta} - \eta| \geq \frac{1-e^{-1}}{A}, 
$$
that is, the condition $F^{-1}_{0}(1-1/(\lfloor d/s_d\rfloor-1)) \le F^{-1}_{1}(1/A)$ on the sequence $(s_d)_{d\ge 1}$ for all $d$ large enough is necessary for almost full recovery. Similarly,
the condition $F^{-1}_{0}(1-1/(\lfloor d/s_d\rfloor-1)) \leq F^{-1}_{1}(1/s_d)$  on the sequence $(s_d)_{d\ge 1}$ for all $d$ large enough is necessary for exact recovery.
\end{corollary}

\subsection{ Minimax and Bayes selectors in general case}
\label{sec:general-lower}

We consider now an alternative approach to the block prior method developed above.
%of subsection \ref{sec:block}. 
It is based on using the uniform prior on the set  of all  $s$-sparse binary vectors $\eta$. 
%Recall that if $\eta\in\Theta_d(s)$, we observe $X_{1},\dots,X_{d}$ where $s$ random variables $X_{i}$ have density $f_{1}$ while the remaining $d-s$ observations have density $f_{0}$. In this section, 
Denote by $\pi$ the uniform distribution on $\Theta_d(s)$, 
that is, we consider that the random vector $\eta=\eta( {\sf S})$ is distributed according to $\pi$ if
\begin{equation}\label{unif}
	\eta_{i}( {\sf S}) = \mathbf{1}(i\in {\sf S}),\quad  i=1,\dots,d, 
\end{equation} 
where ${\sf S}$ is uniformly distributed on $\mathcal{A}:=\{ S \subset \{1,\dots,d\}: |S|=s\}$.   For any $s\ge 2$ and any set $S\in \mathcal{A}$, we define the selector $\eta^{**}$ 
with components
$$
\eta^{**}_{j} = \mathbf{1}\Big\{ f_{1}(X_{j}) G_1(X_i: i\ne j)%\sum_{|S|=s-1,j \not \in S}\Big(\prod_{k  \in S}f_{1}(X_{k})\prod_{k \not \in S}f_{0}(X_{k})\Big) 
\geq  f_{0}(X_{j}) G_0(X_i: i\ne j) %\sum_{|S|=s,j \not \in S}\Big(\prod_{k  \in S}f_{1}(X_{k})\prod_{k \not \in S}f_{0}(X_{k})\Big) 
\Big\}
$$
for $j=1,\dots,d$,
where
$$
G_1(X_i: i\ne j) = \sum_{S\in\mathcal{A}: |S|=s-1,j \not \in S}\Big(\prod_{k  \in S}f_{1}(X_{k})\prod_{k \not \in S}f_{0}(X_{k})\Big),
$$
$$
G_0(X_i: i\ne j) =\sum_{S\in\mathcal{A}:|S|=s,j \not \in S}\Big(\prod_{k  \in S}f_{1}(X_{k})\prod_{k \not \in S}f_{0}(X_{k})\Big).
$$
If all the values $f_0(X_i)$ and $f_1(X_i)$ are positive we can write $\eta_{j}^{**}$ in the form
\begin{equation}\label{eta-star-star}
 \eta_{j}^{**} = \mathbf{1}\left\{ \frac{f_{1}}{f_{0}}(X_{j}) \geq \frac{ \sum_{S: |S|=s, j \not \in S}\prod_{k \in S}\frac{f_{1}}{f_{0}}(X_{k}) }{\sum_{S: |S|=s-1, j \not \in S}\prod_{k \in S}\frac{f_{1}}{f_{0}}(X_{k})}  \right\}   
\end{equation}
for $j=1,\dots,d$.
%We will say that $\eta^{**}$  selects $X_j$ iff $\eta^{**}_j=1$.

\begin{theorem}\label{thm:lower:sparse:two}
Let $s\ge2$. Then
the selector $\eta^{**}$  is minimax on the set $\{\eta: |\eta|_{0} = s\}$ and Bayes with respect to the prior $\pi$:
\begin{align*}
\underset{\tilde{\eta}}{\inf} \underset{|\eta|_{0} = s}{\sup} \mathbf{E}_{\eta} |\tilde{\eta} - \eta| 
&= \underset{|\eta|_{0} = s}{\sup}\mathbf{E}_{\eta} |\eta^{**} - \eta|
= \mathbf{E}_{e(s)} |\eta^{**} - e(s)|,
\\
\underset{\tilde{\eta}}{\inf} \ \mathbb{E}_{\pi} \mathbf{E}_{\eta} |\tilde{\eta} - \eta| 
&= \mathbb{E}_{\pi}\mathbf{E}_{\eta} |\eta^{**} - \eta|
= \mathbf{E}_{e(s)} |\eta^{**} - e(s)|,
\end{align*}
where  $\underset{\tilde{\eta}}{\inf}$ is the infimum over all selectors, $\mathbb{E}_{\pi}$ denotes the expectation over $\eta$ with distribution $\pi$, 
and 
$$e(s)=(\underbrace{1,\dots,1}_{s},\underbrace{0,\dots,0}_{d-s}).
$$ 
\end{theorem}
\begin{proof} As in Theorem \ref{thm:lower:sparse:one} it is enough to conduct the proof by considering $\underset{\tilde{\eta}}{\inf}$ as the infimum over all non-randomized selectors.
We start by observing that 
\begin{align}\label{eq:lower_bound_01a}
\underset{\tilde{\eta}}{\inf} \underset{|\eta|_{0} = s}{\sup} \mathbf{E}_{\eta} |\tilde{\eta} - \eta| &\geq 
\underset{\tilde{\eta}}{\inf}  \, \mathbb{E}_{\pi}\mathbf{E}_{\eta} |\tilde{\eta} - \eta|
\\
&\geq  \sum_{j=1}^{d} \underset{\tilde{\eta}_{j}}{\inf}\, {\rm E}_{\sf S}\mathbf{E}_{\eta( {\sf S})}|\tilde{\eta}_{j} - \mathbf{1}(j\in {\sf S})|.\nonumber
\end{align}
Here,
\begin{eqnarray}\label{eq:lower_bound_2}
&&{\rm E}_{\sf S}\mathbf{E}_{\eta( {\sf S})}|\tilde{\eta}_{j} - \mathbf{1}(j\in {\sf S})|\\
&& =  \frac{1}{|\mathcal{A}|} \sum_{S\in\mathcal{A}: j \in S} \int (1-\tilde{\eta}_{j}(X_1,\dots,X_d)) \Big(\prod_{k  \in S}f_{1}(X_{k})\prod_{k \not \in S}f_{0}(X_{k})\Big) \prod_{k=1}^{d} \mu(dX_k) \nonumber
\\
&& + \frac{1}{|\mathcal{A}|} \sum_{S\in\mathcal{A}: j \not \in S} \int \tilde{\eta}_{j}(X_1,\dots,X_d)  f_0(X_j) \Big(\prod_{k  \in S}f_{1}(X_{k})\prod_{k \not \in S}f_{0}(X_{k})\Big) \prod_{k=1}^{d} \mu(dX_k) \nonumber\\
&&= \frac{1}{|\mathcal{A}|}
%\sum_{S:|S|=s-1, j \not \in S} 
\int (1-\tilde{\eta}_{j}(X_1,\dots,X_d)) 
f_1(X_j) %\Big(\prod_{k  \in S}f_{1}(X_{k})\prod_{k \not \in S}f_{0}(X_{k})\Big) 
G_1(X_i: i\ne j)\prod_{k=1}^{d} \mu(dX_k)
\nonumber
\\
&&\quad + \frac{1}{|\mathcal{A}|} 
%\sum_{S:|S|=s, j \not \in S} 
\int \tilde{\eta}_{j}(X_1,\dots,X_d) f_0(X_j) %\Big(\prod_{k  \in S}f_{1}(X_{k})\prod_{k \not \in S}f_{0}(X_{k})\Big)  
G_0(X_i: i\ne j)\prod_{k=1}^{d} \mu(dX_k) \nonumber.
\end{eqnarray}
Arguing as in the Neyman-Pearson lemma, we conclude that 
$
\eta^{**}_{j}$
minimizes \eqref{eq:lower_bound_2} among all $\{0,1\}$-valued statistics, and hence
\begin{equation}\label{eq:lower_bound_02a}
\underset{\tilde{\eta}}{\inf}\, \mathbb{E}_{\pi}\mathbf{E}_{\eta} |\tilde{\eta} - \eta| \geq  \sum_{j=1}^{d}  {\rm E}_{\sf S}\mathbf{E}_{\eta( {\sf S})}|\eta^{**}_{j} - \mathbf{1}(j\in {\sf S})| = \mathbb{E}_{\pi}\mathbf{E}_{\eta} |\eta^{**} - \eta|.
\end{equation}
Next, notice that, by permutation invariance, the distribution of $|\eta^{**} - \eta|$ under ${\mathbf P}_\eta$ is the same for all $\eta$ such that $|\eta|_{0} = s$. Thus,
\begin{equation}\label{eq:lower_bound_3a}
\mathbb{E}_{\pi}\mathbf{E}_{\eta} |\eta^{**} - \eta|
= \mathbf{E}_{e(s)} |\eta^{**} - e(s)| = \underset{|\eta|_{0} = s}{\sup} \mathbf{E}_{\eta} |\eta^{**} - \eta|.
\end{equation}
The theorem now follows by combining  \eqref{eq:lower_bound_02a} and \eqref{eq:lower_bound_3a}.
\end{proof}

The minimax selector $\eta^{**}$ has a complex form and, in general, it cannot be computed in polynomial time. 
Nevertheless, we can get some insight into its structure.  Intuitively, we would expect a relevant selector in our setting to have sparsity at most $s$, and to select only $X_i$'s corresponding to the largest likelihood ratios $L_i=\frac{f_1}{f_0}(X_i)$.  The next proposition shows that it is indeed the case for the minimax selector $\eta^{**}$. 

In what follows, we will assume that the likelihood ratios $L_i=\frac{f_1}{f_0}(X_i)$ are well-defined $\mathbf{P}_{\eta}$-a.s. for any $\eta$ (this is the case if the support of $f_1$ is included in the support of $f_0$).
Let  $L_{(1)}\ge \cdots \ge L_{(d)}$ be the ordered values of $L_1,\dots,L_d$, and let
\begin{equation}\label{def:Ls}
    \Ls = \{ L_{(1)},\dots,L_{(s)}\}
\end{equation} 
be the set of $s$ largest values in the collection $(L_1,\dots,L_d)$. If $L_i=L_j$ for some $i\ne j$ then any permutation of these two values is allowed in the ordering.

\begin{proposition}\label{proposition:lower:sparse:2}
Let $s\ge2$. Then 
$\eta^{**}$ only selects $X_i$'s corresponding to the largest values $L_i$, i.e., if $L_{(i+1)}$ is selected then $L_{(i)}$ is also selected. 
Moreover,  it holds that
\begin{equation}\label{eq:proposition:lower:sparse:2}
 {\rm supp}(\eta^{**}) \subseteq \{i: L_i\in\Ls\}, 
\end{equation} 
where we denote by $ {\rm supp}(\eta^{**})$ the support of  selector $\eta^{**}$. 
\end{proposition}

% \textcolor{red}{Remark: For the second statement we do not need that the values $L_1,\dots, L_d$ are distinct. See the updated proof.}

The proof of this proposition is given in the Appendix. Remark that it is a deterministic fact valid for any non-negative values $L_1,\dots, L_d$
and a binary vector $\eta^{**}$ with the components
\begin{equation}\label{eta-star-star-L}
\eta_{j}^{**} = \mathbf{1}\Big\{ L_{j} \sum_{S: |S|=s-1, j \not \in S}\prod_{k \in S}L_{k} \geq  \sum_{S: |S|=s, j \not \in S}\prod_{k \in S}L_{k}  \Big\}.  
\end{equation}

It follows from Proposition \ref{proposition:lower:sparse:2} that the components  selected by $\eta^{**}$ are necessarily among those with the $s$ largest values $L_i$. However, the number of components selected by $\eta^{**}$ is not necessarily $s$. %In general, it can be smaller or bigger than $s$. 

We may also notice that $\eta^{**}$ is a data dependent thresholding procedure. However, it seems problematic to get the explicit expression for the threshold.    
On the other hand, Proposition \ref{proposition:lower:sparse:2} suggests that a simple polynomial time selector that chooses only $X_i$'s with the $s$ largest likelihood ratios $L_i$ might be a good approximation for  $\eta^{**}$.
Theorem \ref{th:1/2:general} below shows that it is actually true. Indeed,
consider the estimator
 $\hat{\eta}^{(s)}=(\hat{\eta}^{(s)}_1,\dots,\hat{\eta}^{(s)}_d)$ with the components $$\hat{\eta}^{(s)}_i = \mathbf{1}( L_i \in \Ls)
 \quad i=1\dots,d,$$
 %=\mathbf{1}( X_i \ge X_{(s)})$
 where $\Ls$ is defined in (\ref{def:Ls}). We will call $\hat{\eta}^{(s)}$ the {\it scan selector}. 
The following assumption will be useful below:
\begin{equation}\label{ass:distinct}
 {\mathbf P}_\eta(L_i\ne L_j, \, \forall i\ne j)=1,
 \quad \forall\, \eta: \, |\eta|_{0} = s.  
\end{equation}
This assumption meaning that $L_1,\dots,L_d$ are $\mathbf{P}_\eta$-a.s. distinct is satisfied, for example, if the MLR property holds and $f_0$, $f_1$ are densities with respect to the Lebesgue measure. The Hamming risk of the scan selector can be expressed as follows. 
 \begin{proposition}\label{lem:scan-1}
If \eqref{ass:distinct} holds
then for any $\eta$ such that $|\eta|_{0} = s$ we have
\[
\mathbf{E}_{\eta} |\hat{\eta}^{(s)} - \eta| = 2 s\mathbf{P}_{e(s)}\left( L_{1} \not\in \Ls \right).
\]
\end{proposition}
\begin{proof}
It suffices to prove the proposition for $\eta=e(s)$ as, by permutation invariance,	the distribution of $|\hat{\eta}^{(s)} - \eta|$ under ${\mathbf P}_\eta$ is the same for all $\eta$ such that $|\eta|_{0} = s$. 
Since $L_1,\dots,L_d$ are ${\mathbf P}_\eta$-a.s. distinct (cf. \eqref{ass:distinct}) the vector $\hat{\eta}^{(s)}$ has ${\mathbf P}_\eta$-a.s. exactly $s$ non-zero entries. As $e(s)$ has also exactly $s$ non-zero entries
%Notice that since the vectors $\hat{\eta}^{(s)}$ and $e(s)$ have both exactly $s$ non-zero entries \simo{(they might have more than $s$ if the likelihood values are not distinct)} 
the sets 
$S_{\hat{\eta}^{(s)}}\setminus S_{e(s)}$ and $S_{e(s)}\setminus S_{\hat{\eta}^{(s)}}$ are of the same cardinality ${\mathbf P}_\eta$-a.s. Hence, ${\mathbf P}_\eta$-a.s. we have
\begin{eqnarray}\label{eq:2}
|\hat{\eta}^{(s)} - e(s)| &=& |S_{\hat{\eta}^{(s)}}\setminus S_{e(s)}| + |S_{e(s)}\setminus S_{\hat{\eta}^{(s)}}|
\\
&=&2|S_{e(s)}\setminus S_{\hat{\eta}^{(s)}}| = 2\sum_{i=1}^s \mathbf{1}(L_{i} \not\in \Ls). \nonumber
\end{eqnarray}
Using the fact that the random variables $\mathbf{1}(L_1 \not\in \Ls),\dots, \mathbf{1}(L_s \not\in \Ls)$ are identically distributed under $\mathbf{P}_{e(s)}$ we find:
\[
\mathbf{E}_{e(s)} |\hat{\eta}^{(s)} - e(s)| = 2\mathbf{E}_{e(s)}  \sum_{i=1}^s \mathbf{1}(L_i \not\in \Ls)  = 2s\mathbf{P}_{e(s)} (L_{1} \not\in \Ls).
\]
\end{proof}

 Using Proposition \ref{lem:scan-1} we obtain the following theorem.
\begin{theorem}\label{th:1/2:general}
Let $s\ge 2 $.  If \eqref{ass:distinct} holds then
$$
\frac{1}{2}\mathbf{E}_{e(s)} |\hat{\eta}^{(s)} - e(s)| \leq \underset{\tilde{\eta}}{\inf} \underset{|\eta|_{0} = s}{\sup} \mathbf{E}_{\eta} |\tilde{\eta} - \eta| \leq \mathbf{E}_{e(s)} |\hat{\eta}^{(s)} - e(s)|.
$$
\end{theorem}
\begin{proof}
The second inequality of the theorem is straightforward using the fact that, by permutation invariance, the distribution of $|\hat{\eta}^{(s)} - \eta|$ under ${\mathbf P}_\eta$ is the same for all $\eta$ such that $|\eta|_{0} = s$. 
Hence we only focus on the first inequality. %Note first that for $s=1$ it is proved in Proposition \ref{prop:lower-bound-s=1}. 
For $s\ge 2$, 
using \eqref{eq:lower_bound_01a} and \eqref{eq:lower_bound_02a} we get
\begin{eqnarray*}
\underset{\tilde{\eta}}{\inf} \underset{|\eta|_{0} = s}{\sup} \mathbf{E}_{\eta} |\tilde{\eta} - \eta| 
&\geq&
 \frac{1}{|\mathcal{A}|}\sum_{j=1}^{d}\sum_{S \in \mathcal{A}}\mathbf{E}_{\eta(S)}|\eta^{**}_{j} - \mathbf{1}(j\in S)|\\
 &\geq&
 \frac{1}{|\mathcal{A}|}\sum_{j=1}^{d}\sum_{S \in \mathcal{A}: j\in S}\mathbf{E}_{\eta(S)}|\eta^{**}_{j} - 1|\\
 &=&
 \frac{d}{|\mathcal{A}|}\sum_{S \in \mathcal{A}: 1\in S}\mathbf{E}_{\eta(S)}|\eta^{**}_{1} - 1|
\end{eqnarray*}
since all components $\eta^{**}_{j}$ with $j\in S$ have the same distribution under $\mathbf{P}_{\eta(S)}$ with $S \in \mathcal{A}$.  We now invoke  Lemma \ref{lemma:lower:sparse} from the Appendix to get that, for all $S \in \mathcal{A}$ such that  $1\in S$, 
$$
\mathbf{E}_{\eta(S)}|\eta^{**}_{1} - 1| = \mathbf{P}_{\eta(S)}( \eta_{1}^{**}=0)
\ge 
\mathbf{P}_{\eta(S)}\left(L_1 \leq L_{(s),-1} \right) 
\ge
\mathbf{P}_{\eta(S)}( L_1 \not\in  \Ls ),
$$
where $L_{(s),-1}$ denotes the $s$-th largest value in the collection $(L_2, \dots, L_d)$. %\textcolor{red}{(the last inequality becomes equality if the values $L_1, \dots, L_d$ are distinct $\mathbf{P}_{\eta}$-a.s. for any $\eta$ such that $|\eta|_{0} = s$)}. 
Next, for all $S \in \mathcal{A}$ such that  $1\in S$,  
$$
 \mathbf{P}_{\eta(S)}( L_1 \not\in  \Ls )= \mathbf{P}_{e(s)}( L_1 \not\in  \Ls )
$$
due to the fact that the distribution of $\Ls$ under such $\mathbf{P}_{\eta(S)}$ is invariant under permutations of $(L_2, \dots, L_d)$.   Combining the above inequalities we find
$$
\underset{\tilde{\eta}}{\inf} \underset{|\eta|_{0} = s}{\sup} \mathbf{E}_{\eta} |\tilde{\eta} - \eta| 
\geq 
\frac{d}{|\mathcal{A}|}\sum_{S\in \mathcal{A}: 1 \in S} \mathbf{P}_{e(s)}( L_1 \not\in  \Ls )
= 
\frac{d  \binom{d-1}{s-1} }{\binom{d}{s}} \mathbf{P}_{e(s)}( L_1 \not\in  \Ls ). % = s \mathbf{P}_{e_{1,s}}( X_{1}<X_{(s),-1} ).
$$
Since
$
\frac{d \binom{d-1}{s-1}}{\binom{d}{s}} = s
$
the result now follows using Proposition \ref{lem:scan-1}.
\end{proof}

It follows from Theorems \ref{thm:lower:sparse:two} and \ref{th:1/2:general} that the scan selector $\hat{\eta}^{(s)}$ achieves both the minimax risk and the Bayes risk to within factor $2$. 

Interestingly, if the error criterion is not the expected Hamming risk but the probability of wrong recovery ${\mathbf P}_\eta(S_{\widehat\eta}  \ne S_\eta)$  then  the scan selector turns out to be minimax optimal. Indeed, we have the following theorem.
\begin{theorem}\label{th:minimax:prob:wrong:recovery}
If \eqref{ass:distinct} holds then we have
\begin{align}\label{eq1:th:minimax:prob:wrong:recovery}
  \underset{\tilde{\eta}}{\inf} \underset{|\eta|_{0} = s}{\sup}
{\mathbf P}_\eta(S_{\tilde\eta}  \ne S_\eta) 
&= \underset{|\eta|_{0} = s}{\sup}
{\mathbf P}_\eta(S_{\hat{\eta}^{(s)}}  \ne S_\eta)
= 
{\mathbf P}_{e(s)}(S_{\hat{\eta}^{(s)}}  \ne S_{e(s)}),
\\
\label{eq2:th:minimax:prob:wrong:recovery}
  \underset{\tilde{\eta}}{\inf} \ \mathbb{E}_{\pi}
{\mathbf P}_\eta(S_{\tilde\eta}  \ne S_\eta) 
&= \ \mathbb{E}_{\pi}
{\mathbf P}_\eta(S_{\hat{\eta}^{(s)}}  \ne S_\eta)
= 
{\mathbf P}_{e(s)}(S_{\hat{\eta}^{(s)}}  \ne S_{e(s)}).
\end{align}
Moreover, the minimax/Bayes risk here can be expressed as follows:
\begin{equation}\label{eq3:th:minimax:prob:wrong:recovery}
{\mathbf P}_{e(s)}(S_{\hat{\eta}^{(s)}}  \ne S_{e(s)})
= 
{\mathbf P}_{e(s)}\Big(\bigcup_{i=1}^s \{L_{i} \not\in \Ls\}\Big).
\end{equation}
\end{theorem}
\begin{proof}
	The first equality in \eqref{eq2:th:minimax:prob:wrong:recovery} is proved
in \cite[Lemma 5.3]{gao2019fundamental}. The second equalities in \eqref{eq1:th:minimax:prob:wrong:recovery} and  in \eqref{eq2:th:minimax:prob:wrong:recovery} are due to the representation ${\mathbf P}_\eta(S_{\hat{\eta}^{(s)}} \ne S_{\eta}) = {\mathbf P}_\eta(|\hat{\eta}^{(s)} - \eta|\ge 1)$ (cf. \eqref{eq:risks}) and the
fact that, by permutation invariance, the distribution of $|\hat{\eta}^{(s)} - \eta|$  under ${\mathbf P}_\eta$ is the same for all $\eta$ such that $|\eta|_{0} = s$. Next, the first equality in \eqref{eq1:th:minimax:prob:wrong:recovery} is obtained from these remarks and the inequalities
$$
\underset{\tilde{\eta}}{\inf} \ \mathbb{E}_{\pi}
{\mathbf P}_\eta(S_{\tilde\eta}  \ne S_\eta) \le \underset{\tilde{\eta}}{\inf} \underset{|\eta|_{0} = s}{\sup}
{\mathbf P}_\eta(S_{\tilde\eta}  \ne S_\eta) \le 
\underset{|\eta|_{0} = s}{\sup}
{\mathbf P}_\eta(S_{\hat{\eta}^{(s)}}  \ne S_\eta).
$$
Thus, \eqref{eq1:th:minimax:prob:wrong:recovery}  and \eqref{eq2:th:minimax:prob:wrong:recovery} follow.
In order to prove  \eqref{eq3:th:minimax:prob:wrong:recovery} it suffices to note that, by virtue of   \eqref{eq:risks}  and \eqref{eq:2},  
\begin{align*}
{\mathbf P}_{e(s)}(S_{\hat{\eta}^{(s)}}  \ne S_{e(s)})
&	= 	{\mathbf P}_{e(s)}(|\hat{\eta}^{(s)} - {e(s)}|\ge 1 )
\\
&
= {\mathbf P}_{e(s)}\Big(2\sum_{i=1}^s \mathbf{1}(L_{i} \not\in \Ls)\ge 1\Big)
\\
&
= {\mathbf P}_{e(s)}\Big(\bigcup_{i=1}^s \{L_{i} \not\in \Ls\}\Big).
\end{align*}
%equality  is due to  From Theorem~\ref{th:minimax:prob:wrong:recovery},
\end{proof}

Finally, we establish a lower bound that will be useful in what follows. 

\begin{theorem}\label{th:lower-both-risks}
	If \eqref{ass:distinct} holds then
	\begin{align}\label{eq:th:lower-both-risks}
	\underset{\tilde{\eta}}{\inf} \underset{|\eta|_{0} = s}{\sup} \mathbf{E}_{\eta} |\tilde{\eta} - \eta| & \ge
	\underset{\tilde{\eta}}{\inf} \underset{|\eta|_{0} = s}{\sup}
	{\mathbf P}_\eta(S_{\tilde\eta}  \ne S_\eta) 
	\\ &
	\ge \nonumber
	\mathbf{P}_{e(s)}\Big(\underset{j=1,\dots,s}{\min}\,L_{j} \leq \underset{j=s+1,\dots,d}{\max} L_j \Big). 
	\end{align}
\end{theorem}
\begin{proof}
The first inequality in \eqref{eq:th:lower-both-risks} follows immediately from \eqref{eq:risks}.  Thus, we prove only the second inequality.
Using~\eqref{eq3:th:minimax:prob:wrong:recovery}
we get
\begin{align*}%\label{eq1:th:lower-both-risks}
	\underset{\tilde{\eta}}{\inf} \underset{|\eta|_{0} = s}{\sup}
	{\mathbf P}_{e(s)}(S_{\tilde\eta}  \ne S_{e(s)}) 
	&
	 = {\mathbf P}_{e(s)}\Big(\bigcup_{i=1}^s \{L_{i} \not\in \Ls\}\Big)
	 \\
	 &
	\ge  \mathbf{P}_{e(s)}\Big( \underset{j=1,\dots,s}{\min}\,L_{j} \not\in \Ls \Big).
\end{align*}
The proof is completed by observing that since $L_1,\dots,L_d$ are ${\mathbf P}_{e(s)}$-a.s. distinct (cf. \eqref{ass:distinct})
 the condition $\underset{j=1,\dots,s}{\min}\,L_{j} \leq \underset{j=s+1,\dots,d}{\max} L_j$  implies that $\underset{j=1,\dots,s}{\min}\,L_{j} \not\in \Ls$
 (${\mathbf P}_{e(s)}$-a.s.).
\end{proof}

Although the selector $\hat{\eta}^{(s)}$ can be computed in  polynomial time, it has a drawback to require, in general case, the knowledge of densities $f_0$ and $f_1$ that may not be accessible or have a complicated form. This problem disappears if the MLR property holds, as discussed in the next subsection.

\subsection{Consequences for MLR models}
\label{subsec:MLR}

We assume now that the MLR property is satisfied. 
Let  $X_{(1)}\ge \cdots \ge X_{(d)}$ be the ordered values of $X_1,\dots,X_d$, and let
\begin{equation}\label{S-star}
    \Xs = \{ X_{(1)},\dots,X_{(s)}\}
\end{equation} 
be the set of $s$ largest values in the sample $(X_1,\dots,X_d)$. If $X_i=X_j$ for some $i\ne j$, then any permutation of these two values is allowed in the ordering. However, all $X_i$'s are almost surely distinct if $f_0$ and $f_1$ are densities with respect to the Lebesgue measure.

Clearly, under the MLR property,  the $i$th component of the scan selector  takes the form
$\hat{\eta}^{(s)}_i = \mathbf{1}( X_i \in \Xs)$,
where $\Xs$ is given in (\ref{S-star}).  Therefore, the scan selector is one and the same for all MLR models. In other words, it is adaptive to $f_0$ and $f_1$ under the MLR assumption. 
Using the results of Section \ref{sec:general-lower} we now establish some properties of the selector 
\begin{align}\label{def:scan-selector-underMLR}
 \hat{\eta}^{(s)}_*:=( \mathbf{1}( X_1 \in \Xs),\dots,  \mathbf{1}( X_d \in \Xs)).   
\end{align}
\begin{proposition}\label{lem:scan}
Let $X_1,\dots,X_d$ be ${\mathbf P}_\eta$-a.s. distinct for any $\eta$ such that $|\eta|_{0} = s$.
	Then for any $\eta$ such that $|\eta|_{0} = s$ we have
	\[
	\mathbf{E}_{\eta} |\hat{\eta}^{(s)}_* - \eta| = 2 s\mathbf{P}_{e(s)}\left( X_1 \not\in \Xs \right).
	\]
\end{proposition}

We emphasize that Proposition \ref{lem:scan} does not require the MLR property. Its proof follows the same lines as the proof of Proposition \ref{lem:scan-1}
with the only difference that $L_i$ and $\Ls$ are replaced by $X_i$ and $\Xs$. 

The next theorem is an immediate consequence 
of Theorem \ref{th:1/2:general}.

\begin{theorem}\label{prop:lower:quantile}
Let $X_1,\dots,X_d$ be ${\mathbf P}_\eta$-a.s. distinct for any $\eta$ such that $|\eta|_{0} = s$.
Assume that $s\ge 2$ and  the MLR property holds. Then
$$
\frac{1}{2}\mathbf{E}_{e(s)} |\hat{\eta}^{(s)}_* - e(s)| \leq \underset{\tilde{\eta}}{\inf} \underset{|\eta|_{0} = s}{\sup} \mathbf{E}_{\eta} |\tilde{\eta} - \eta| \leq \mathbf{E}_{e(s)} |\hat{\eta}^{(s)}_* - e(s)|.
$$
\end{theorem}

%\simo{treat case of equality for more than $s$ terms.}

%
%As an application, observe that when component $1$ belongs to the true support of the signal, then
%$$
%X_{(s),-1} \geq Z_{(s)},
%$$
%where $Z_{(s)}$ is the $s$-largest component not in the support (noise component). Hence, and based on the fact that $Z_{(s)}$ can be well approximated by $F_{0}^{-1}(1-\frac{s}{d-s})$, we recover the same result as in Proposition \ref{proposition:lower:sparse:1} through block sparsity.

	We also have the following lower bound.
\begin{corollary}\label{cor:exact_rec}
Let $X_1,\dots,X_d$ be ${\mathbf P}_\eta$-a.s. distinct for any $\eta$ such that $|\eta|_{0} = s$.
Assume that  the MLR property holds. Then
\begin{align}\label{eq:cor:exact_rec}
	\underset{\tilde{\eta}}{\inf} \underset{|\eta|_{0} = s}{\sup} \mathbf{E}_{\eta} |\tilde{\eta} - \eta| & \ge
	\underset{\tilde{\eta}}{\inf} \underset{|\eta|_{0} = s}{\sup}
	{\mathbf P}_\eta(S_{\tilde\eta}  \ne S_\eta) 
	\\ &
	\ge \nonumber
	\mathbf{P}_{e(s)}\Big(\underset{j=1,\dots,s}{\min}\,X_{j} \leq \underset{j=s+1,\dots,d}{\max} X_j \Big). 
	\end{align}
Furthermore, under the assumption $F^{-1}_{0}(1-1/(d-s)) > F^{-1}_{1}(1/s)$ we have
\begin{equation}\label{eq:cor:exact_rec1}
\underset{\tilde{\eta}}{\inf} \underset{|\eta|_{0} = s}{\sup} \mathbf{E}_{\eta} |\tilde{\eta} - \eta|
\ge
	\underset{\tilde{\eta}}{\inf} \underset{|\eta|_{0} = s}{\sup}
	{\mathbf P}_\eta(S_{\tilde\eta}  \ne S_\eta) \geq (1-e^{-1})^2,
\end{equation}
that is, the condition $F^{-1}_{0}(1-1/(d-s_d)) \leq F^{-1}_{1}(1/s_d)$  on the sequence $(s_d)_{d\ge 1}$ for all $d$ large enough is necessary for exact recovery.
%the condition $F^{-1}_{0}(1-1/(d-s)) \leq F^{-1}_{1}(1/s)$ is necessary for exact recovery.
\end{corollary}

The bound \eqref{eq:cor:exact_rec} follows immediately from Theorem \ref{th:lower-both-risks} and the MLR property. 
The proof of \eqref{eq:cor:exact_rec1} is given in the Appendix.

Note that the necessary condition $F^{-1}_{0}(1-1/(d-s_d)) \leq F^{-1}_{1}(1/s_d)$  of exact recovery based on Corollary \ref{cor:exact_rec} provides a stronger result than the condition $F^{-1}_{0}(1-s_d/(d-s_d)) \le F^{-1}_{1}(1/s_d)$ obtained via the block prior technique of Section \ref{sec:block} (cf.  Corollary \ref{cor:almost_rec}). 

The lower bound \eqref{eq:cor:exact_rec} has a nice interpretation as it tells us that exact recovery is impossible as long as the event, where the distributions on the support of $e(s)$ and outside of the support cannot be separated, is not trivial. Considering only $e(s)$ is without loss of generality since, by permutation invariance, the r.h.s. of \eqref{eq:cor:exact_rec} can be expressed in terms of $\mathbf{P}_\eta$ and the corresponding quantities for any $\eta\in\Theta_d(s)$.

Finally, the following proposition allows us to relate the risk of the selector 
$\hat{\eta}^{(s)}_*$ to the risk of any separable thresholding selector. It will be useful in the examples considered below in Sections \ref{sec:log-concave} and \ref{sec:grouped}.
\begin{proposition}\label{prop:relation-to-thresholding-lambda} Let $X_1,\dots,X_n$ be distinct real numbers and $\lambda>0$. Consider the thresholding selector
$
\bar{\eta}_\lambda:=( \mathbf{1}( X_1>\lambda),\dots,  \mathbf{1}( X_d >\lambda)).
$
For any $\eta$ such that $|\eta|_0=s$ we have
 $$|\hat{\eta}^{(s)}_* - \eta| \le 2|\bar{\eta}_\lambda - \eta|.$$
\end{proposition}
Noteworthy, Proposition \ref{prop:relation-to-thresholding-lambda} is a deterministic fact. The inequality holds directly for the losses and not only for the risks, and it is true for arbitrary threshold $\lambda>0$. Moreover, as  $X_i$'s can be any real numbers, the result is valid, for example, if we take $L_i$'s instead of $X_i$'s in the definition of $\bar{\eta}_\lambda$, and we replace $\hat{\eta}^{(s)}_*$ by $\hat{\eta}^{(s)}$.

\subsection{Connection to the risk of the best separable selector
}\label{sec:lower_bound_impossible}

Throughout this subsection, we assume that the MLR property is satisfied, 
	and $f_0$, $f_1$ are densities with respect to the Lebesgue measure on $\mathbf{R}$. In particular, it implies \eqref{ass:distinct}. We also assume that $s\ge 2$.

We first introduce some notation. Let $t_{1}$ be the unique solution of the 
equation $$
(s-1)F_{1}(x) = (d-s)(1-F_{0}(x)).
$$
We also assume that there exists a solution $t_{2}$ of the
equation
$$
(s-1)f_{1}(x) = (d-s)f_{0}(x).
$$
Note that if $t_{2}$ exists, it is unique due to the MLR property. We finally define the function $\Psi:\mathbf{R}\to \mathbf{R}$ with values
$$
\quad \Psi(x)= (s-1)F_{1}(x)+(d-s)(1-F_{0}(x)), \quad  x \in \mathbf{R}.
$$
Notice that $\Psi(t_{2})$ is very close to the risk of the best separable selector $\Psi_{\rm sep}(d,s)$ defined in \eqref{eq:old_bound}.  
The following proposition relates $\Psi(t_{1})$, $\Psi(t_{2})$ and $\Psi_{\rm sep}$.
\begin{proposition}\label{lem:comparison}
	Assume that $s\ge2$. Let the MLR property be satisfied, let
	$f_0$, $f_1$ be densities with respect to the Lebesgue measure on $\mathbf{R}$, and let $t_{1},t_{2},\Psi(\cdot)$ be as defined above. Then
	\begin{align}\label{eq1:lem:comparison}
	\Psi(t_{2})&= \Psi_{\rm sep}(d-1,s-1), \\
\Psi(t_{2}) &\leq \Psi(t_{1}) \leq 2\Psi(t_{2}). \label{eq2:lem:comparison}
	\end{align}
\end{proposition}
Thus, the values $\Psi(t_{1})$ and $\Psi(t_{2})$ can be viewed as tight characterizations of the risk of the best separable selector.
Furthermore, we show that they determine the minimax risk to within numerical constants in the zone where exact recovery is impossible. Indeed, the following lower bound holds.

\begin{theorem}\label{theorem3}
Let the assumptions of Proposition \ref{lem:comparison} be satisfied, and $2\le s \leq (d+2)/3$, $\Psi(t_1) \geq 24$. Then
$$
\underset{\tilde{\eta}}{\inf} \underset{|\eta|_{0} = s}{\sup} \mathbf{E}_{\eta} |\tilde{\eta} - \eta| \geq \frac{1}{20}\Psi(t_{1}) \geq \frac{1}{20}\Psi(t_{2})= \frac{1}{20} \Psi_{\rm sep}(d-1,s-1).
$$
\end{theorem}
On the other hand, using \eqref{eq:minimax-separable} we find that 
\begin{equation*}
{\underset{\tilde{\eta}}{\inf} \underset{|\eta|_{0} = s}{\sup} \mathbf{E}_{\eta} |\tilde{\eta} - \eta| \le     
\Psi_{\rm sep}(d,s).}
\end{equation*} 
Hence the result of Theorem \ref{theorem3} is non-asymptotic minimax optimal in the regime where $\Psi(t_1)$ is greater than an absolute constant (that is, where exact recovery is not possible). 

To summarize,  in general the minimax risk is given  by the risk of the scan selector $\hat{\eta}^{(s)}$ to within factor 2. Moreover, under the MLR and Lebesgue density assumptions, in the regime where exact recovery is impossible, the minimax risk is characterized by  the explicit quantity, which is the risk of the best separable selector.   It remains an interesting open question to derive a similar kind of characterization  in the regime where exact recovery is possible.

\section{Variable selection in light tail location models}
\label{sec:log-concave}
\setcounter{equation}{0}

In this section we apply the results of Section \ref{sec:minimax-and-bayes} to variable selection in the model
\begin{align}\label{model:log-concave}
X_i = \mu_i + \sigma \xi_i, \quad \text{for} \quad i=1,\dots,d,
\end{align}
where $\sigma>0$ and the random variables $\xi_i$ are i.i.d. with c.d.f. $F$ belonging to the following class of light tail distributions
\begin{align*}
\mathcal{P}_\nu &= \{ F: \R \to [0,1] , F \text{ is a c.d.f. such that } 1-F(u) \leq e^{-|u|^{\nu}/\nu}  \\ & \qquad
\text{ and }  F(-u) \leq e^{-|u|^{\nu}/\nu},\ \forall u \geq 0\}
\end{align*}
indexed by parameter $\nu \ge 1$. Let $a>0$. We assume that the mean parameter $\mu=(\mu_1,\dots,\mu_d) $ belongs to the set
\begin{align*}
\Theta_{d} (s,a) = \{& \mu \in \R^{d}: \text{there exists } S \subset \{1,\dots,d\} \text { with } |S| = s, \\
&\mu_j \geq a \text { for } j \in S  \text { and } \mu_j =0 \text { for } j \not \in S\}.
\end{align*}
The problem of variable selection under this model is stated as follows. 
For each vector $\mu\in\Theta_{d} (s,a)$, we consider the set $S_\mu$ of its non-zero components and define the vector $\eta(\mu) = (\eta_1(\mu), \ldots, \eta_d(\mu))$, where $\eta_j(\mu) = \mathbf{1} ( j \in S_\mu)$.
Variable selection consists in  estimating the binary vector $\eta(\mu)$  based on the observations $(X_1,\dots,X_d)$ satisfying \eqref{model:log-concave}. 
We evaluate the performance of a selector $\widehat \eta$ by the Hamming distance $|\widehat \eta  - \eta(\mu)| = \sum_{j=1}^d |\widehat \eta_j - \eta_j(\mu)|$.  As above, we use the expected Hamming risk and the probability of wrong recovery as risk measures.  

For a sequence $(s_d,a_d)_{d \geq 1}$, we say that a selector $\widehat \eta$ achieves exact recovery for $\Theta_{d}(s_d,a_d)$  if 
\begin{align}\label{exact-rec-lc}
\limsup_{d\to \infty} \underset{F \in \mathcal{P}_\nu}{\sup}  \ \sup_{\mu \in \Theta_{d} (s_d,a_d)} \, {\mathbf E}_{\mu,F} |\widehat \eta - \eta(\mu)| =0
\end{align}
and it achieves almost full recovery for $\Theta_{d}(s_d,a_d)$  if
\begin{align}\label{almost-full-rec-lc}
\limsup_{d \to \infty} \underset{F \in \mathcal{P}_\nu}{\sup} \ \sup_{\mu \in \Theta_{d} (s_d,a_d)} \, \frac{1}{s_d} {\mathbf E}_{\mu,F} |\widehat \eta - \eta(\mu)| =0,
\end{align}
where ${\mathbf E}_{\mu,F}$ denotes the expectation with respect to the probability measure ${\mathbf P}_{\mu,F}$ of the observations  $(X_1,\dots,X_d)$ satisfying \eqref{model:log-concave}.  An analog of exact recovery property \eqref{exact-rec-lc} can be also stated in a weaker form, with the probability of wrong recovery rather than the expected Hamming risk:
\begin{align}\label{exact-rec-lc-PWR}
\limsup_{d\to \infty} \underset{F \in \mathcal{P}_\nu}{\sup}  \ \sup_{\mu \in \Theta_{d} (s_d,a_d)} \, {\mathbf P}_{\mu,F} (S_{\widehat \eta} \ne S_{\eta(\mu)}) =0.
\end{align}
The aim of this section is to characterize the range of sequences $(a_d)$, for which exact recovery and almost full recovery properties defined in 
\eqref{exact-rec-lc}, \eqref{almost-full-rec-lc} and \eqref{exact-rec-lc-PWR} are possible. We establish what is usually referred to as {\it sharp
necessary and sufficient conditions of recovery}.  Namely, in each case we find the critical sequence $a_d^*\to \infty$ such that there exists a selector achieving the corresponding form of recovery if $a_d\ge a_d^*(1+\alpha_d)$ for some $\alpha_d\to 0$, and no selector can achieve it if $a_d< a_d^*(1+\beta_d)$ for some $\beta_d\to 0$ as $d\to \infty$.  If this holds, we will say that $a_d^*$ realizes the sharp phase transition.

Analogous model was recently studied in \cite{gao2019fundamental} and \cite{abraham2021sharp} where, instead of dealing with the class of c.d.f.'s  $\mathcal{P}_\nu$, the authors considered fixed c.d.f. $F$ with light tail density satisfying similar condition on the tail (Subbotin type condition).
The results in \cite{abraham2021sharp} provide sharp necessary and sufficient conditions of recovery with respect to the multiple testing risk rather than  to the expected Hamming risk and the probability of wrong recovery  that we consider here. %in \eqref{exact-rec-lc}, \eqref{almost-full-rec-lc} and \eqref{exact-rec-lc-PWR}.
The paper \cite{gao2019fundamental} studies the asymptotics of the probability of wrong recovery.  Under the additional polynomial growth assumption $s_d\sim d^{\beta}$ 
%=\lfloor d^{1-\beta} \rfloor$ 
for some fixed 
$\beta$ in (0,1), it establishes 
sufficient conditions for exact recovery  and proves that these conditions are necessary within the set of thresholding selectors. 

The following non-asymptotic minimax lower bounds imply the necessary conditions of recovery over all selectors with no polynomial growth restriction.  
\begin{theorem}\label{prop:app_1}
Let $\nu > 1$. There exist positive constants $c(\nu), c'(\nu)$ depending only on $\nu$ with $c'(\nu)\ge \min \{t\ge e: \nu\log(t) > c(\nu)\log\log(t)\}$ such that for 
$\min(d-s,s)>c'(\nu)$ and 
\[
a <\sigma \left( (\nu\log (d-s) - c(\nu)\log\log (d-s) )^{1/\nu} + (\nu\log(s) - c(\nu)\log\log(s) )^{1/\nu} \right)
\]
we have
\begin{align}\label{eq:prop:app_1}
\underset{\tilde{\eta}}{\inf}\underset{F \in \mathcal{P}_\nu}{\sup} \ \underset{\mu \in \Theta_{d} (s,a)}{\sup} \mathbf{E}_{\mu,F} |\tilde{\eta} - \eta(\mu)| 
&\ge
\underset{\tilde{\eta}}{\inf} \underset{F \in \mathcal{P}_\nu}{\sup} \ \underset{\mu \in \Theta_{d} (s,a)}{\sup}
{\mathbf P}_{\mu,F} (S_{\tilde\eta}  \ne S_{\eta(\mu)}) 
\\ &
\ge
(1-e^{-1})^2. \nonumber
\end{align}
Furthermore, if for some $A>1$ we have $\min(d/s-1,A)>c'(\nu)$ and 
$$
a < \sigma \left((\nu\log (d/s-1) - c(\nu)\log\log (d/s-1) )^{1/\nu} + (\nu\log(A) - c(\nu)\log\log(A) )^{1/\nu} \right),
$$
then
\begin{align}\label{eq1:prop:app_1}
\underset{\tilde{\eta}}{\inf} \underset{F \in \mathcal{P}_\nu}{\sup}  \ \underset{\mu \in \Theta_{d} (s,a)}{\sup} \frac{1}{s}\mathbf{E}_{\mu,F} |\tilde{\eta} - \eta(\mu)| 
	\geq \frac{1-e^{-1}}{A}.
\end{align}
\end{theorem}

For  $\nu \geq 1$ and $x>e$, we define the thresholding selector $\hat{\eta}_{\nu}(x)$ with components 
\begin{equation}\label{def:selector:lc}
  \hat{\eta}_{\nu,i}(x) = \mathbf{1}(X_i \geq \sigma(\nu \log(x) + \nu \log\log(x))^{1/\nu} ), \quad i=1,\dots,d.  
\end{equation}
The following theorem gives sufficient conditions for both exact and almost full recovery.
\begin{theorem}\label{prop:app_2}
Let $\nu\ge 1$. Let the sequences $(s_d,a_d)_{d\geq 1}$ be such that $d -s_d\to \infty$ and
\[
a_d \geq \sigma \left( (\nu \log(d-s_d) + \nu \log\log(d-s_d))^{1/\nu} + (\nu \log(s_d) + \nu \log\log(d-s_d))^{1/\nu} \right)
\]
for all $d$ large enough. Then the selector $\widehat\eta = \hat{\eta}_{\nu}(d-s_d)$ 
%and the scan selector $\widehat\eta =\hat{\eta}^{(s_d)}_*$ defined in  \eqref{def:scan-selector-underMLR} satisfy 
satisfies the exact recovery properties \eqref{exact-rec-lc} and  \eqref{exact-rec-lc-PWR}.

If the sequences $(s_d,a_d)_{d\geq 1}$ are such that $d/s_d\to \infty$ and
\[
a_d \geq  \sigma \left((\nu \log(d/s_d-1) + \nu \log\log(d/s_d-1))^{1/\nu} + ( \nu \log\log(d/s_d-1))^{1/\nu} \right) 
\]
for all $d$ large enough, then the selector $\widehat\eta = \hat{\eta}_{\nu}(d/s_d-1)$ 
%and the scan selector $\widehat\eta =\hat{\eta}^{(s_d)}_*$ defined in  \eqref{def:scan-selector-underMLR} satisfy 
satisfies the almost full recovery property \eqref{almost-full-rec-lc}.
\end{theorem}
Theorems \ref{prop:app_1} and \ref{prop:app_2} imply that, if $d-s_d\to \infty$, $s_d\to \infty$
the sharp phase transition for exact recovery (both in the sense of the expected Hamming risk and of the probability of wrong recovery) occurs at $$a_d^{*} = \sigma\left((\nu \log(d-s_d))^{1/\nu}+(\nu \log(s_d))^{1/\nu}\right).$$
For almost full recovery, the sharp phase transition occurs at $a_d^{*} = \sigma(\nu \log(d/s_d))^{1/\nu} $ as long as $d/s_d \to \infty$. Interestingly, the phase transition for almost full recovery turns out to be the same as obtained in \cite{abraham2021sharp} for the multiple testing risk criterion.

If we assume in addition that all $F$ in $\mathcal{P}_\nu$ are absolutely continuous with respect to the Lebesgue measure then the results of Theorem \ref{prop:app_2} hold true for the scan selector $\widehat\eta =\hat{\eta}^{(s_d)}_*$ defined in  \eqref{def:scan-selector-underMLR}. Indeed, in that case $X_1,\dots,X_d$ are $\mathbf{P}_{\mu,F}$-a.s. distinct for any $\mu\in\Theta_{d}(s,a)$ and the same conclusions as in Theorem \ref{prop:app_2} for
 the scan selector follow by applying Proposition \ref{prop:relation-to-thresholding-lambda} and recalling that   \eqref{def:selector:lc} is a thresholding selector.

\section{Grouped variable selection}\label{sec:grouped}
\setcounter{equation}{0}

In this section, we apply the results of Section \ref{sec:minimax-and-bayes} to derive necessary and sufficient conditions of exact and almost full recovery in a model with group sparsity.
%The necessary conditions will be based on the results obtained above. The sufficient conditions will be obtained by an explicit selector construction.
We consider the following Gaussian matrix model:
\begin{equation}\label{model}
Y_{ij}=\theta_{ij} + \sigma \xi_{ij}, \quad \text{for }i=1,\dots,k, \text{ and } j=1,\dots,d,
\end{equation}
where the random variables $ \xi_{ij}$ are i.i.d. with standard normal distribution and $\sigma > 0$.

The mean parameter $\theta \in \R^{k \times d}$ is assumed to belong to the class 
\begin{eqnarray*}
  \Theta_{k,d} (s,a) &=& \{ \theta \in \R^{k \times d}: \text{there exists } S \subset \{1,\dots,d\} \text { with } |S| = s, \\
  &&\| \theta _{j}\| \geq a \text { for } j \in S  \text { and } \theta_{j} =0 \text { for } j \not \in S\},
\end{eqnarray*}
where $s,k,d$ are positive integers, $s < d$, $a>0$ is a real number, $\| \cdot\|$ is the Euclidean norm, and let $\theta_j\in \mathbf{R}^k $ be the $j$th column vector of matrix $\theta$. %, that is, the vector in $$ with components $(\theta_{1j},\dots,\theta_{kj})^\top$. 
Analogously, we denote by $Y_j$ and $\xi_j$ the $j$th column vectors of matrices $(Y_{ij})$ and $( \xi_{ij})$, respectively. 
The class $\Theta_{k,d} (s,a)$ contains all matrices, for which $s$ columns have Euclidean norm at least $a$, and the remaining $d-s$ columns are null vectors.
Here, parameter $s$ plays the role of sparsity.

For each $\theta\in\Theta_{k,d} (s,a)$, we set $S(\theta) =  \{ j \in \{1,\dots,d\}:  \theta_{j} \not =0 \}$ and define the vector $\eta(\theta) = (\eta_1(\theta), \ldots, \eta_d(\theta))$, where $\eta_j(\theta) = \mathbf{1} ( j \in S(\theta))$.
 In this section, we study the problem of selecting the set $S(\theta)$ of non-null columns of $\theta$ based on the observations $(Y_{ij})$. 
        A column selector $\widehat \eta = (\widehat \eta_1, \ldots, \widehat \eta_d)$ is any measurable function taking values in $\{0,1\}^d$.
We evaluate the performance of $\widehat \eta$ by the Hamming distance $|\widehat \eta  - \eta(\theta)| = \sum_{j=1}^d |\widehat \eta_j - \eta_j(\theta)|$. 

Similarly to the setting considered in the previous sections, for a sequence $(s_d,k_d,a_d)_{d \geq 1}$ we will say that a selector $\widehat \eta$ achieves exact recovery for $\Theta_{k_d,d}(s_d,a_d)$  if 
\begin{align}\label{exact-rec}
\limsup_{d\to \infty} \sup_{\theta \in \Theta_{k_d,d} (s_d,a_d)} \, {\mathbf E}_\theta |\widehat \eta - \eta(\theta)| =0
\end{align}
 and it achieves almost full recovery for $\Theta_{k_d,d}(s_d,a_d)$ 
\begin{align}\label{almost-full-rec}
\limsup_{d \to \infty} \sup_{\theta \in \Theta_{k_d,d} (s_d,a_d)} \, \frac{1}{s_d} {\mathbf E}_\theta |\widehat \eta - \eta(\theta)| =0,
\end{align}
where ${\mathbf E}_\theta$ denotes the expectation with respect to the probability measure ${\mathbf P}_\theta$ of the observations  $(Y_{ij})$ satisfying \eqref{model}. As above, we also consider the definition of exact recovery in a weaker form:
\begin{align}\label{exact-rec-PWR}
\limsup_{d\to \infty} \sup_{\theta \in \Theta_{k_d,d} (s_d,a_d)} \, {\mathbf P}_\theta (S_{\widehat \eta} \ne S_{\eta(\theta)}) =0.
\end{align}
The main question we are addressing below is to find in what range of values of $a_d$ the exact recovery and the almost full recovery are possible.

%%%%%%%%%%%%%%%%%%%%%
\subsection{Reduction by rotation invariance}
%%%%%%%%%%%%%%%%%%%%%%

 We first show that, based on rotation invariance, we can reduce the group variable selection problem stated above to a simpler problem.
 Indeed, the next lemma shows that it suffices to consider selectors depending only on the norms $\|Y_{j}\| $ of the observed random vectors $Y_{j}$, $j=1,\dots,d$.

\begin{lemma}\label{lem:reduction} For any selector $\hat \eta ( Y_1,\dots,Y_d) = (\hat \eta_1( Y_1,\dots,Y_d), \dots, \hat \eta_d( Y_1,\dots,Y_d))$ there exists a randomized selector $\overline \eta( \|Y_1\|,\dots,\|Y_d\| )$ such that
\begin{align}\label{eq:lem:reduction1}
\sup_{\theta \in \Theta_{k,d} (s,a)} {\mathbf E}_\theta |\hat \eta - \eta (\theta)| \geq 
\sup_{\theta \in \Theta_{k,d} (s,a)} {\mathbf E}_\theta |\overline \eta - \eta (\theta)| 
\end{align}
and 
\begin{align}\label{eq:lem:reduction2}
\sup_{\theta \in \Theta_{k,d} (s,a)} {\mathbf P}_\theta (S_{\widehat \eta} \ne S_{\eta(\theta)}) \geq 
\sup_{\theta \in \Theta_{k,d} (s,a)} {\mathbf P}_\theta (S_{\overline \eta} \ne S_{\eta(\theta)}). 
\end{align}
\end{lemma}

Set $X_j = \|Y_j\|^2$, $j=1,\dots,d$. It follows from Lemma \ref{lem:reduction} that it suffices to consider the simplified model where one observes independent random variables $X_j$, $j=1,\dots,d$, such that 
$s$ among them (with $j\in S(\theta)$) have  non-central $\chi^2$-distributions with $k$ degrees of freedom  and non-centrality parameters $\|\theta_j\|\ge a$, while the remaining $X_j$ (for $j\not \in S(\theta)$) have a central $\chi^2$-distribution with $k$ degrees of freedom. In the sequel, we assume this model.

%%%%%%%%%%%%%%%%%%%%%%
\subsection{Lower bound for the minimax risk %Necessary conditions
}
%%%%%%%%%%%%%%%%%%%%%%

The following theorem provides lower bounds on the minimax risk in the group variable setting implying the  necessary conditions of exact and almost full recovery.
\begin{theorem}\label{th:lower-bound:group} There exist positive absolute constants $c_1$
	 and $c_2$ such that the following holds. If $d-s\ge c_2$ and
\[
a^{2} < c_1\sigma^2 (\log(d-s) \vee \sqrt{k\log(d-s)})
\]
then 
\begin{align}\label{eq:exact-rec-lower}
 \inf_{\tilde\eta} \sup_{\theta \in \Theta_{k,d} (s,a)} \, {\mathbf E}_\theta |\tilde \eta - \eta(\theta)|\ge 
 \inf_{\tilde\eta} \sup_{\theta \in \Theta_{k,d} (s,a)} \, {\mathbf P}_\theta (S_{\tilde \eta} \ne S_{\eta(\theta)})\ge (1-e^{-1})^2.
\end{align}
If $d/s\ge c_2$ and 
\[
a^{2} < c_1\sigma^2(\log((d-s)/s) \vee \sqrt{k\log((d-s)/s)})
\]
then
\begin{align}\label{eq:almost-full-rec-lower}
\inf_{\tilde\eta} \sup_{\theta \in \Theta_{k,d} (s,a)} \,  \frac1s{\mathbf E}_\theta |\tilde \eta - \eta(\theta)|\ge \frac{ 1-e^{-1}}{2}. 
\end{align}
\end{theorem}
The proof of this theorem is given in the Appendix. It is based on  considering the smaller parameter set 
$$\Theta'_{k,d} (s,a)=%\cap 
\Big\{\theta\in\Theta_{k,d} (s,a) :\,\|\theta_j\|\in \{0,a\}, j=1,\dots,d\Big\}
$$
instead of $\Theta_{k,d} (s,a)$. If $\theta$ belongs to $\Theta'_{k,d} (s,a)$  then $s$ among the observations $X_j=\|Y_j\|^2$ are distributed with density $f_1$ of non-central $\chi^2_k$-distribution with non-centrality parameter $a$ while the remaining $d-s$ observations have  a central $\chi^2_k$-distribution with density denoted by $f_0$.
Thus, we have a special case of pivotal selection problem considered in Section \ref{sec:minimax-and-bayes} and we apply Corollaries \ref{cor:almost_rec} and \ref{cor:exact_rec} to obtain the lower bounds. 

%%%%%%%%%%%%%%%%%%%%%
\subsection{Optimal selector and sufficient conditions of recovery%upper bound for the risk
}
%%%%%%%%%%%%%%%%%%%%%%

Intuitively, it is clear that one should select the components $j$ such that the value $X_j=\|Y_j\|^2$ is large enough. The same conclusion can be obtained by analyzing the corresponding pivotal selection problem with parameter set $\Theta'_{k,d} (s,a)$. Recall that the best separable selector for the pivotal problem has the form \eqref{eq:separable}, where $f_0$ and $f_1$ are now the central and non-central $\chi^2_k$ densities, respectively. Comparing the explicit expressions for these densities (cf. equations \eqref{f0} and \eqref{f1} in Appendix \ref{appendix:B}) we see that the  likelihood ratio $f_1(z)/f_0(z)$ is monotone in $z$. Therefore, the best separable selector \eqref{eq:separable} can be written in the form
\begin{equation}\label{opt-selector}
 \widehat \eta_j = \mathbf{1} ( X_j \geq t_d), \quad  j=1,\dots,d,   
\end{equation}
with an appropriate definition of threshold $t_d>0$. The next theorem shows that such a selector, as well as the scan selector \eqref{def:scan-selector-underMLR} achieve almost full and exact recovery under the conditions that asymptotically match the lower bounds of Theorem \ref{th:lower-bound:group}.

\begin{theorem}\label{th:upper-bound:group}
Let the sequence $(s_d,k_d, a_d)_{d\geq 1}$ be such that
\[
a_d^2 \geq \sigma^2(16\sqrt{k_d\log(d)} + 80 \log(d))
\]
for all $d$ large enough.
Then the selector \eqref{opt-selector} with $t_d= \sigma^2(k_d + 4 \log(d) + 4 \sqrt{k_d\log(d)})$ and the scan selector $\widehat\eta =\hat{\eta}^{(s_d)}_*$ defined in  \eqref{def:scan-selector-underMLR}
satisfy the exact recovery properties \eqref{exact-rec} and  \eqref{exact-rec-PWR}. 

 Let the sequences $(s_d,k_d, a_d)_{d\geq 1}$ be such that $d/s_d\to \infty$ and
\[
a_d^2 \geq \sigma^2(16\sqrt{k_d\log(d/s_d)} + 80 \log(d/s_d))
\]
then the selector \eqref{opt-selector} with $t_d= \sigma^2(k_d + 4 \log(d/s_d) + 4 \sqrt{k_d\log(d/s_d)})$ and the scan selector $\widehat\eta =\hat{\eta}^{(s_d)}_*$ defined in  \eqref{def:scan-selector-underMLR}
satisfy the almost full recovery property \eqref{almost-full-rec}.
\end{theorem}
Combining Theorems \ref{th:lower-bound:group} and \ref{th:upper-bound:group} we conclude that the phase transition occurs at the critical threshold $a^{*}_d  \asymp  \sigma \max((\log d)^{1/2}, (k_d\log d)^{1/4}) $ for exact recovery, and at $a^{*}_d \asymp  \sigma \max((\log (d/s_d))^{1/2}, (k_d\log  (d/s_d))^{1/4})$
for almost full recovery as long as $d/s_d \to \infty$. The results hold with no restriction on the values $k_d$ that can vary in an arbitrary way as $d$ increases. 

Note that the selector \eqref{opt-selector} can be restated as a group Lasso selector for our model. For general linear regression models, group Lasso selection of sparsity pattern was considered in \cite{lounici2010oracle}. Being specified to the case of our model, the condition of exact recovery from \cite{lounici2010oracle} requires  $a_d$ to be at least of the order $\max((\log d)^{1/2}, k_d^{1/2})$,  which is suboptimal in view of the above results.

%%%%%%%%%%%%%%%%%%%%%%%

 \section*{Acknowledgement}
The work of Cristina Butucea and Alexandre Tsybakov was supported by the French National Research Agency (ANR) under the grant Labex Ecodec (ANR-11-LABEX-0047). The work of Mohamed Ndaoud is supported by a Chair of
Excellence in Data Science granted by the CY Initiative.
\bibliographystyle{alpha}
%\bibliography{ref.bib}

\appendix
\section{Proofs for Section \ref{sec:minimax-and-bayes}}\label{appendix:A}
\renewcommand{\theequation}{\thesection.\arabic{equation}}
\setcounter{equation}{0}

%\noindent\textbf{Proof of Proposition \ref{prop:lower-bound-s=1}.}

\noindent\textbf{Proof of Proposition \ref{proposition:lower:sparse:1}.}
We partition the set $\{1,\dots,s\lfloor d/s\rfloor\}$ into $s$ equal blocks $B_{1},\dots,B_{s}$, each of size $\lfloor d/s\rfloor$. Denote by $\eta_{B_{i}}$ and $\tilde{\eta}_{B_{i}}$ the restrictions of $\eta$ and $\tilde{\eta}$ to $B_i$. Let $\eta_{B_{i}}$ be distributed according to the uniform prior $\pi_{B_{i}}$ on $B_{i}$, that is, only one entry of $\eta_{B_{i}}$ is equal to 1 and each entry has equal probability to be drawn. Thus, $|\eta_{B_{i}}|_{0}=1$.  We consider the product prior $\pi=\prod_{i=1}^s \pi_{B_{i}}$ on $(\eta_1, \dots, \eta_{s\lfloor d/s\rfloor})$.  In words, the prior is such that each block contains exactly one non-zero entry, which is sampled uniformly in the block, and all blocks are independent. We have
$$
\underset{\tilde{\eta}}{\inf} \underset{|\eta|_{0} = s}{\sup} \mathbf{E}_{\eta} |\tilde{\eta} - \eta|
\geq 
\underset{\tilde{\eta}}{\inf}\,\mathbb{E}_{\pi}\mathbf{E}_{\eta}\sum_{i=1}^{s}|\tilde{\eta}_{B_{i}} - \eta_{B_{i}}|
\geq 
\sum_{i=1}^{s}\underset{\tilde{\eta}_{B_{i}}}{\inf}\,\mathbb{E}_{\pi}\mathbf{E}_{\eta}|\tilde{\eta}_{B_{i}} - \eta_{B_{i}}|.
$$
Note that here $\tilde{\eta}_{B_{i}}$ is a function of all $X_1,\dots,X_d$. 
Since the blocks are independent, we get using conditional means and Jensen's inequality that for any $\tilde{\eta}_{B_{i}}$ there exists $\hat{\eta}_{B_{i}}=\hat{\eta}_{B_{i}}(X_{B_i})$ measurable with respect to only $X_{B_i}:=(X_j: j\in B_i)$
such that $\mathbb{E}_{\pi}\mathbf{E}_{\eta}|\tilde{\eta}_{B_{i}} - \eta_{B_{i}}|\ge \mathbb{E}_{\pi}\mathbf{E}_{\eta}|\hat{\eta}_{B_{i}}(X_{B_i}) - \eta_{B_{i}}|$. Therefore,
$$
\underset{\tilde{\eta}}{\inf} \underset{|\eta|_{0} = s}{\sup} \mathbf{E}_{\eta} |\tilde{\eta} - \eta| \geq \sum_{i=1}^{s}\underset{\hat{\eta}_{B_{i}}}{\inf}\,\mathbb{E}_{\pi}\mathbf{E}_{\eta}|\hat{\eta}_{B_{i}}(X_{B_i}) - \eta_{B_{i}}|,
$$
where the infima are now taken over selectors depending only on observations in the corresponding block. Applying Proposition \ref{prop:lower-bound-s=1} for each block completes the proof. 

%\noindent\textbf{Proof of Theorem \ref{thm:lower:sparse:two}.}
 %\begin{proof}[]
 %In order to prove Theorem \ref{th:1/2:general}, need the following lemma. 
 
We now come to the proof of Proposition \ref{proposition:lower:sparse:2}. 
 As mentioned after the statement of the proposition we can assume that $L_1,\dots,L_d$ are some deterministic values and that $\eta^{**}$ is a deterministic vector with  components
chosen according to \eqref{eta-star-star-L}. For the proof of the proposition we will need the following lemma.

\begin{lemma}\label{lemma:lower:sparse}
Let $s\ge2$. Then
\begin{align}\label{eq1:lemma:lower:sparse}
   \sum_{S:|S|=s, j \not \in S}\prod_{k \in S} L_k \ %\frac{f_{1}}{f_{0}}(X_{k})   
   > \
   L_{(s),-j} %\frac{f_{1}}{f_{0}}(X_{(s),-j})
   \sum_{S:|S|=s-1, j \not \in S}\prod_{k \in S} L_k %\frac{f_{1}}{f_{0}}(X_{k}) 
\end{align}
for $j= 1,\dots,d,$ where $L_{(s),-j}$ denotes the $s$-th largest value in the collection $(L_{i}: i \neq j)$.
\end{lemma}

 \noindent\textbf{Proof of Lemma \ref{lemma:lower:sparse}.}
Let $j$ be in $\{1,\dots,d\}$ and let $\hat{S}_{j}$ be the set of indices of the  $s$ largest values in the collection $(L_{i}: i \neq j)$ where perhaps only a subset of indices $i$ with $L_i = L_{(s),-j}$ is put into $\hat{S}_{j}$. Choose a fixed element $j^*$ of $\hat{S}_{j}$. Obviously, for any set $\tilde S\subset \{1,\dots,d\}$  of size $s-1$, 
the intersection between $\tilde S^{c}$ and $\hat{S}_{j}$ is non-empty. For each $\tilde{S}$, we choose an element $i(\tilde{S})$ of this intersection.  If $j^*$ is not an element of $\tilde{S}$ we put $i(\tilde{S})= j^*$. We have that $i(\tilde{S})\ne j$ and
$$
\{ S = \tilde S \cup \{i(\tilde S)\}: |\tilde S|=s-1 , j \not \in \tilde S  \} \subsetneqq  \{ S: |S|=s, j \not \in S\}.
$$
Note that the inclusion is strict because the set on the left hand side does not contain sets that do not have $j^*$ as an element.
It follows that
\begin{align*}
  \sum_{S:|S|=s, j \not \in S}\prod_{k \in S}L_{k}\
  & >
  \sum_{\tilde{S}:|\tilde{S}|=s-1, j \not \in \tilde{S}}L_{i(\tilde{S})}\prod_{k \in \tilde{S}}L_{k} \\
  &
  \ge
  L_{(s),-j}\sum_{\tilde{S}:|\tilde{S}|=s-1, j \not \in \tilde{S}}\prod_{k \in \tilde{S}}L_{k}, 
\end{align*}
which shows the claim of the lemma.

%\textcolor{red}{
%The following part of the proof can be omitted: Indeed, due to the fact that each element of $\{L_i:i\in \hat{S}_j\}$ is not smaller than $L_{(s),-j}$ we have
%\begin{align}\label{eq:lem5}
%  L_{i(\tilde{S})}
%  &
%  \ge L_{(s),-j},\quad \forall \tilde{S}:|\tilde{S}|=s-1, j \not \in \tilde{S}. 
%\end{align}
%Moreover, if  the values $L_1,\dots,L_d$ are distinct then 
%the inequality in \eqref{eq:lem5} is strict for the set  $\tilde{S}$, which is the set of indices of the $s-1$ smallest values among $(L_i: i\in \hat{S}_j)$. Indeed, in this case, $L_{i(\tilde{S})}= \max_{i\in \hat{S}_j} L_i$ and hence  $L_{i(\tilde{S})}>L_{(s),-j}$. 
%}

\noindent\textbf{Proof of Proposition \ref{proposition:lower:sparse:2}.}
To prove the first assertion of the proposition, we assume w.l.o.g that $X_{2}$ is selected and we show that if $L_1\ge L_2$ then $X_{1}$ is also selected. Since $X_{2}$ is selected we have
\begin{equation}\label{star}
L_{2}  \sum_{|S|=s-1, 2 \not \in S}\prod_{k \in S}L_{k}  \geq  \sum_{|S|=s, 2 \not \in S}\prod_{k \in S}L_{k}     .
\end{equation}
Set
$$T_{1,2} = \sum_{|S|=s, 1 \in S, 2 \in S}\prod_{k \in S}L_{k},
$$
$$T_{-1,2} = \sum_{|S|=s, 1 \not \in S, 2 \in S}\prod_{k \in S}L_{k},
$$
$$
T_{1,-2} = \sum_{|S|=s, 1 \in S, 2 \not \in S}\prod_{k \in S}L_{k},
$$
$$
T_{-1,-2} = \sum_{|S|=s, 1 \not \in S, 2 \not \in S}\prod_{k \in S}L_{k}.
$$
Then we can write \eqref{star} in the form
$$
T_{1,2} + T_{-1,2} \geq T_{1,-2} + T_{-1,-2}.
$$
Since $L_{1}\geq L_{2}$ we have %due to the MLR property that $\frac{f_{1}}{f_{0}}(X_{1}) \geq \frac{f_{1}}{f_{0}}(X_{2})$ and thus 
$
T_{-1,2} \leq T_{1,-2}.
$
It follows that
$$
T_{1,2} + T_{1,-2} \geq T_{-1,2} + T_{-1,-2},
$$
so that $X_1$ is selected. This proves  the first assertion of the proposition.

Due to Lemma \ref{lemma:lower:sparse}, the bound \eqref{eq1:lemma:lower:sparse} holds with strict inequality. This fact  and the definition of $\eta^{**}$, cf. \eqref{eta-star-star-L}, imply that if $\eta^{**}_{j} =1$ then
$
L_{j}  >
L_{(s),-j}$. On the other hand, $L_{(s),-j} \geq L_{(s+1)}$. It follows that $j \in \{i: L_i\in \Ls\}$.
%\simo{Here we need distinct values because being strictly larger that the $s+1$ order statistic does not imply belonging to $\Ls$. It only means that the value is amongst the $s$ largest ones. We can make all proofs work if we replace the definition of $\Ls$ by the set of values of $L_i$ larger than $L_{(s)}$ and in that case it could be a set of size larger than $s$. I this stage I would suggest to keep the hypothesis of distinct values.}
Thus, \eqref{eq:proposition:lower:sparse:2} is proved. 

%\noindent\textbf{Proof of Proposition \ref{lem:scan}.}

\noindent\textbf{Proof of Corollary \ref{cor:exact_rec}.}
We need to prove only \eqref{eq:cor:exact_rec1} since \eqref{eq:cor:exact_rec} follows from Theorem \ref{th:lower-both-risks} and the MLR property.
If $F^{-1}_{0}(1-1/(d-s)) > F^{-1}_{1}(1/s)$ we get from \eqref{eq:cor:exact_rec} that
\begin{align*}
&\underset{\tilde{\eta}}{\inf} \underset{|\eta|_{0} = s}{\sup} 
{\mathbf P}_\eta(S_{\tilde\eta}  \ne S_\eta) %\mathbf{E}_{\eta} |\tilde{\eta} - \eta| 
\\
&\quad \geq \mathbf{P}_{e(s)}\Big( \underset{j=1,\dots,s}{\min}\,X_{j}  \leq F^{-1}_{1}(1/s) \Big)\mathbf{P}_{e(s)}\Big( \underset{j=s+1,\dots,d}{\max}\,X_{j}  \geq F^{-1}_{0}(1-1/(d-s)) \Big).
\end{align*}
Here,
\begin{align*}
	\mathbf{P}_{e(s)}\Big( F^{-1}_{0}(1-1/(d-s) \leq \underset{j=s+1,\dots,d}{\max}\,X_{j}  \Big) &= 1 - [F_{0}(F^{-1}_{0}(1-1/(d-s))]^{d-s}\\
	& =1 - (1-1/(d-s))^{d-s} \geq 1-e^{-1},
\end{align*}
and
\begin{align*}
\mathbf{P}_{e(s)}\Big(   \underset{j=1,\dots,s}{\min}\,X_{j}  \leq F^{-1}_{1}(1/s) \Big) &= 1 - (1-F_{1}(F^{-1}_{1}(1/s))^{s} 
\\
&= 1 - (1-1/s)^{s} \geq 1-e^{-1}.
\end{align*}
We conclude that under the assumption $F^{-1}_{0}(1-1/(d-s)) > F^{-1}_{1}(1/s)$ we have
\eqref{eq:cor:exact_rec1}
and thus the exact recovery is impossible.

\noindent\textbf{Proof of Proposition \ref{prop:relation-to-thresholding-lambda}.}
Set for brevity  $S^*=S_{\eta}$, $\hat S= S_{\hat{\eta}^{(s)}_*}=\{i: X_i\in \Xs\}$, and $\tilde S = S_{\bar{\eta}_\lambda}= \{i: X_i> \lambda\}$. Since $X_i$'s are distinct the set $\hat S$ has cardinality $s$. Thus, $S^*$ and $\hat S$ being of the same cardinality we have $|S^*\setminus \hat S|=|\hat S\setminus S^*|$, and
$$
|\hat{\eta}^{(s)}_* - \eta|= |S^*\setminus \hat S|
+|\hat S\setminus S^*|=2|S^*\setminus \hat S| = 2|\hat S\setminus S^*|.
$$
On the other hand, 
$$
|\bar{\eta}_\lambda - \eta|= |S^*\setminus \tilde S|+
|\tilde S\setminus S^*|.
$$
We consider now the two cases. First, if $\lambda \ge X_{(s)}$ then $\tilde S\subseteq \hat S$, so that 
\begin{align*}
 |\bar{\eta}_\lambda - \eta|\ge |S^*\setminus \tilde S| &\ge   |S^*\setminus \hat S|=\frac12|\hat{\eta}^{(s)}_* - \eta|.
\end{align*}
Second, if $\lambda < X_{(s)}$ then $\hat S\subseteq \tilde S$, and we have
\begin{align*}
 |\bar{\eta}_\lambda - \eta|\ge |\tilde S\setminus S^*| \ge   |\hat S\setminus S^*|=\frac12|\hat{\eta}^{(s)}_* - \eta|.
\end{align*}

\noindent\textbf{Proof of Proposition \ref{lem:comparison}.}
We prove only \eqref{eq2:lem:comparison} since \eqref{eq1:lem:comparison} is an easy consequence of the definitions. 
We first prove the right hand inequality in \eqref{eq2:lem:comparison}. Observe that
$$
\Psi(t_{1}) = 2(s-1)F_{1}(t_{1}) = 2(d-s)(1-F_{0}(t_{1})).
$$
If $t_{1}\geq t_{2}$ then
$$
\Psi(t_{1}) \leq 2(d-s)(1-F_{0}(t_{2})) \leq 2\Psi(t_{2}),
$$
and if $t_{1}< t_{2}$ then
$$
\Psi(t_{1}) < 2(s-1)F_{1}(t_{2}) \leq 2\Psi(t_{2}).
$$
We now prove the left hand inequality in \eqref{eq2:lem:comparison}. We have
$$
d-s - \Psi(x) = (d-s)F_{0}(x)-(s-1)F_{1}(x).
$$
Hence
$$
d-s - \Psi(t_{2}) = \int_{-\infty}^{t_{2}} [(d-s)f_{0}(u)-(s-1)f_{1}(u)]du.
$$
Using the definition of $t_{2}$ and the MLR property it is not hard to check  that
$$
d-s - \Psi(t_{2}) = \underset{t\in \mathbf{R}}{\sup} \int_{-\infty}^{t} [(d-s)f_{0}(u)-(s-1)f_{1}(u)]du. %\underset{A}{\sup} \int_{A}(d-s)f_{0}(u)-(s-1)f_{1}(u)du.
$$
Hence,
$$
d-s - \Psi(t_{2}) \geq d-s - \Psi(t_{1}),
$$
which completes the proof.

\noindent\textbf{Proof of Theorem \ref{theorem3}.}
Let  $X_{(s),-1}$ the $s$-largest component of $(X_2,\dots,X_d)$. Under the assumptions of Theorem \ref{theorem3} the values $X_1,\dots,X_d$ are distinct  $\mathbf{P}_{\eta}$-a.s. for any $\eta$. Therefore, the condition $X_{1} \leq X_{(s),-1}$ is a.s. equivalent to $X_{1} < X_{(s),-1}$, which implies that $X_1\not\in \Xs$.  Thus, $\mathbf{P}_{e(s)}\left( X_1\not\in \Xs \right)\ge \mathbf{P}_{e(s)}\left( X_{1} \leq X_{(s),-1} \right)$. 
	Combining this inequality with Theorem \ref{prop:lower:quantile}  and Proposition \ref{lem:scan} we get
	$$
	\underset{\tilde{\eta}}{\inf} \underset{|\eta|_{0} = s}{\sup} \mathbf{E}_{\eta} |\tilde{\eta} - \eta| \geq s\mathbf{P}_{e(s)}\left( X_{1} \leq X_{(s),-1} \right) \ge s\mathbf{P}_{e(s)}\left( X_{(s),-1} \ge t_1 \right)\mathbf{P}_{e(s)}\left( X_{1} \leq t_{1}\right).
	$$
	We now invoke the following lemma.
	\begin{lemma}\label{lem:quantile}
		Under the assumptions of Theorem \ref{theorem3}
		$$
		\mathbf{P}_{e(s)}\left( X_{(s),-1} \geq t_{1} \right) \geq \frac{1}{10}.
		$$
	\end{lemma}
	%Using Lemma \ref{lem:quantile}, we get moreover that
	It follows that
	$$
	\underset{\tilde{\eta}}{\inf} \underset{|\eta|_{0} = s}{\sup} \mathbf{E}_{\eta} |\tilde{\eta} - \eta| \geq \frac{s}{10}\mathbf{P}_{e(s)}\left( X_{1} \leq t_{1}\right)\geq\frac{s-1}{10}F_{1}(t_{1}).
	$$
	By the definition of $t_{1}$ we have $\Psi(t_{1})=2(s-1)F_{1}(t_{1})$. Thus,
	$$
	\underset{\tilde{\eta}}{\inf} \underset{|\eta|_{0} = s}{\sup} \mathbf{E}_{\eta} |\tilde{\eta} - \eta|\geq\frac{1}{20}\Psi(t_{1}).
	$$
	We conclude using Proposition \ref{lem:comparison}.

	\noindent\textbf{Proof of Lemma \ref{lem:quantile}.}
	We use the following fact:
	$$
	\mathbf{P}_{e(s)}\left( X_{(s),-1} \geq t_{1} \right)  = \mathbf{P}_{e(s)}\left( \sum_{i=2}^{d}\mathbf{1}{\{X_{i} \geq t_{1}\}} \geq s\right).
	$$
	Under $\mathbf{P}_{e(s)}$, exactly $s-1$ random variables among $X_{2},\dots,X_{d}$ are distributed with density $f_{1}$ and the remaining $d-s$ have density $f_0$.  
	Then 
	$$
	\mathbf{E}_{e(s)}\left( \sum_{i=2}^{d}\mathbf{1}{\{X_{i} \geq t_{1}\}} \right) = (s-1)(1-F_{1}(t_{1})) + (d-s)(1-F_{0}(t_{1}))= s-1.
	$$
	It follows that
	$$
	\mathbf{P}_{e(s)}\left( X_{(s),-1} \geq t_{1} \right) = \mathbf{P}_{e(s)}\left( \sum_{i=2}^{d}\mathbf{1}{\{X_{i} \geq t_{1}\}} -\mathbf{E}_{e(s)}\left( \sum_{i=2}^{d}\mathbf{1}{\{X_{i} \geq t_{1}\}} \right) \geq 1\right).
	$$
	Next, the variance under $\mathbf{P}_{e(s)}$ is
	\[
	V:= \V_{e(s)} \left( \sum_{i=2}^{d}\mathbf{1}{\{X_{i} \geq t_{1}\}}  \right) = (s-1)F_{1}(t_1)(1-F_{1}(t_1)) + (d-s)F_{0}(t_1)(1-F_{0}(t_1)).
	\]
	Using the fact that $\Psi(t_1) = 2(s-1)F_{1}(t_1) = 2(d-s)(1-F_{0}(t_1))$ we get 
	\begin{align*}
	V &= \frac{\Psi(t_1)}{2}(1 - F_{1}(t_1) +  F_{0}(t_1)) \\
	&= \frac{\Psi(t_1)}{2}\left(2 - \Psi(t_1)\frac{d-1}{2(s-1)(d-s)}\right).
	\end{align*}
	{Since $\Psi(t_1) \leq 2(s-1)$, we have
		\[
		V \geq \frac{\Psi(t_1)}{2} \left(2 - \frac{d-1}{d-s} \right) \geq \frac{\Psi(t_1)}{4}
		\]
	}
	%for some $\kappa$ in (0,1), if $s \leq ((1-\kappa)d + 1)/(2-\kappa)$.
	{if $s \leq (d + 2)/3$.}
	%How does this change the Berry-Esseen below?}
	Using the Berry–Esseen theorem in  \cite{shevtsova} 
	%Shevtsova, Irina (2010). "An Improvement of Convergence Rate Estimates in the Lyapunov Theorem". Doklady Mathematics. 82 (3): 862–864
	we obtain
	\[
	\mathbf{P}_{e(s)}\left( X_{(s),-1} \geq t_{1} \right) \geq \mathbf{P}(\mathcal{N}(0,1) \geq {2/\sqrt{\Psi(t_1)}}) - 0.56\frac{W}{V^{3/2}},
	\]
	where 
	\[
	W =  \sum_{i=2}^{d}\mathbf{E}_{e(s)}\left(|
	\mathbf{1} {\{X_{i} \geq t_{1}\}} - \mathbf{P}_{e(s)}(X_i \geq t_{1})|^3 \right) \leq V. 
	\]
	{Since $\Psi(t_1) \geq 24$ we have $V\ge6$,} and hence 
	$$
	\mathbf{P}_{e(s)}\left( \sum_{i=2}^{d}\mathbf{1}{\{X_{i} \geq t_{1}\}} \geq s\right) \geq 
	{\mathbf{P}(\mathcal{N}(0,1) \geq 1/\sqrt{6}) - 0.56/\sqrt{6}}  \ge \frac{1}{10},
	$$
	{where we have observed that $\mathbf{P}(\mathcal{N}(0,1) \geq 1/\sqrt{6})\geq 0.34$ and $0.56/\sqrt{6}\leq 0.23$. }

\section{Proofs for Sections \ref{sec:log-concave} and \ref{sec:grouped}}\label{appendix:B}
\setcounter{equation}{0}

\noindent\textbf{Proof of Theorem \ref{prop:app_1}.}
Consider the following subset of $\Theta_{d} (s,a)$:
$$\Theta'_{d} (s,a)=\Theta_{d} (s,a)\cap \Big\{\mu \in \mathbb{R}^d:\, \mu_j \in \{0,a\}, j=1,\dots,d\Big\}.
$$
Fix $\nu > 1$ and denote by $\bar F$ the Subbotin c.d.f. with density $\bar f=\bar F'$ such that
$$
\bar f(x)=\frac{\nu^{1-1/\nu}}{2\Gamma(1/\nu)} e^{-|x|^\nu/\nu}, \quad x\in \mathbf{R}.
$$
Note that $\bar F$ belongs to $\mathcal{P}_\nu$. Indeed, since $\nu\ge 1$ we have $
\bar f(x)\le \frac{1}{2} e^{-|x|^\nu/\nu},
$
while $\int_u^\infty e^{-x^\nu/\nu} dx \le  u^{1-\nu} e^{-u^\nu/\nu} $ for any $u>0$.  Also, $\bar F (u)\ge  \frac{1}{2}$ for any $u\ge 0$. It follows that $1-\bar F (u)\le  \frac{1}{2} \min(1, u^{1-\nu}e^{-u^\nu/\nu})\le e^{-u^\nu/\nu}$ for any $u\ge 0$. 

Next, we have
\begin{align}\nonumber
\underset{\tilde{\eta}}{\inf}\underset{F \in \mathcal{P}_\nu}{\sup} \ \underset{\mu \in \Theta_{d} (s,a)}{\sup} \mathbf{E}_{\mu,F} |\tilde{\eta} - \eta(\mu)| 
&\ge
\underset{\tilde{\eta}}{\inf}  \ \underset{\mu \in \Theta'_{d} (s,a)}{\sup}
\mathbf{E}_{\mu,\bar F} |\tilde{\eta} - \eta(\mu)| 
\\ &
\ge       \label{log-concave-reduction}
\underset{\tilde{\eta}}{\inf} \ \underset{\mu \in \Theta'_{d} (s,a)}{\sup} {\mathbf P}_{\mu,\bar F} (S_{\tilde\eta}  \ne S_{\eta(\mu)}).
\end{align}
Set $f_0(\cdot)=\bar f(\cdot/\sigma)/\sigma$. Note that the expressions in \eqref{log-concave-reduction} containing $\underset{\mu \in \Theta'_{d} (s,a)}{\sup}$ are the minimax risks for the pivotal selection problem with densities $f_0(\cdot)$ and $f_1(\cdot)=f_0(\cdot-a)$ supported on $\mathbf{R}$. Since for $\nu > 1$ the function $\log f_0(\cdot)$ is strictly concave it follows that $f_1/f_0(\cdot)$ is monotone increasing, cf., e.g., \cite[Proposition 2.3]{saumard-wellner}\footnote{Proposition 2.3 in \cite{saumard-wellner} states that log-concavity for a location family implies monotone non-decreasing likelihood ratio. By a trivial modification of its proof, strict log-concavity implies monotone increasing likelihood ratio.}. 
Thus, we can apply the results of Section \ref{sec:minimax-and-bayes} for the pivotal selection problem under the MLR property, namely, Corollaries \ref{cor:almost_rec} and \ref{cor:exact_rec}.

%Namely, we use Corollary \ref{cor:exact_rec} to obtain a condition on $a$, under which exact recovery is impossible,
% and Corollary \ref{cor:almost_rec} to obtain the corresponding result for almost full recovery. 
 
 In view of Corollary \ref{cor:exact_rec}, for exact recovery is enough to check that the condition 
 \[
 a <\sigma \left( (\nu\log (d-s) - c(\nu)\log\log (d-s) )^{1/\nu} + (\nu\log s - c(\nu)\log\log s )^{1/\nu} \right)
 \]
implies that $F^{-1}_{0}(1-1/(d-s)) > F^{-1}_{1}(1/s)$
  or, equivalently, 
  \begin{equation*}
  	a < \sigma ({\bar F}^{-1}(1 - 1/(d-s)) +  {\bar F}^{-1}(1-1/s)),
  \end{equation*}
  where we used the facts that $F^{-1}_1(\cdot) = a + F_0^{-1}(\cdot)$,  $ F_0^{-1}(\cdot)=\sigma{\bar F}^{-1}(\cdot) $ and, by the symmetry of the Subbotin distribution, $- {\bar F}^{-1}(1/s)= {\bar F}^{-1}(1-1/s)$.  Thus, to prove \eqref{eq:prop:app_1} it suffices to show that the bound 
  \begin{equation*}%\label{eq:subbotin-lower}
  {\bar F}^{-1}(1 - 1/t) \ge (\nu\log(t) - c(\nu)\log\log(t) )^{1/\nu} = : u^*
  \end{equation*}
  holds if $t>0$ is large enough and to apply it for $t=d-s$ and $t=s$. Note that there exists a constant $c_1(\nu)>0$ depending only on $\nu$ such that $1-  {\bar F}(v)\ge c_1(\nu) v^{1-\nu}e^{-v^\nu/\nu}$ for all $v>0$ large enough. Indeed, it is enough to observe that for $v \geq 1$ we have
  \[ v^{1-\nu} e^{-v^\nu /\nu} = \int_v^\infty (1+(\nu -1)/u^\nu)e^{-u^\nu/\nu} \mathrm{d}u \leq 2\nu^{1/\nu} \Gamma(1/\nu)(1-  {\bar F}(v)).
  \]
  Given this property, it is not hard to check that $1-  {\bar F}(u^*)\ge 1/t$ if $c(\nu)>0$ is large enough. This completes the proof of \eqref{eq:prop:app_1}.

The proof for \eqref{eq1:prop:app_1} (almost full recovery) uses Corollary \ref{cor:almost_rec}, according to which it is enough to check that condition $F^{-1}_{0}(1-1/(\lfloor d/s\rfloor-1)) > F^{-1}_{1}(1/A)$ is satisfied under the given assumption on $a$. This is done by the same argument as above with the only difference that $t=d-s$ and $t=s$ are replaced by $t=\lfloor d/s\rfloor-1$ and $t=1/A$, respectively.

\noindent\textbf{Proof of Theorem \ref{prop:app_2}.}
%\textcolor{blue}{
It is enough to prove the exact recovery and almost full recovery properties \eqref{exact-rec-lc}  and \eqref{almost-full-rec-lc} 
%
%for the selectors $\widehat\eta=\hat{\eta}_{\nu}(d-s_d)$ and $\widehat\eta=\hat{\eta}_{\nu}(d/s_d-1)$, respectively. Indeed,  this implies the same properties 
%for the scan selector $\widehat\eta =\hat{\eta}^{(s_d)}_*$ in view of Proposition \ref{prop:relation-to-thresholding-lambda} and the fact that $\hat{\eta}_{\nu}(d-s_d)$ and $\hat{\eta}_{\nu}(d/s_d-1)$ are thresholding selectors. Moreover, 
since
\eqref{exact-rec-lc-PWR} is an immediate consequence of \eqref{exact-rec-lc} in view of  \eqref{eq:risks}.
%}

Let $F$ be in $\mathcal{P}_\nu$ and  $\mu \in \Theta_{d}(s_d,a_d)$. Without loss of generality assume that $\sigma=1$. We first prove the exact recovery  property  \eqref{exact-rec-lc} for the selector $\widehat\eta=\hat{\eta}_{\nu}(d-s_d)$. Set $t=(\nu \log(d-s_d) + \nu \log\log(d-s_d))^{1/\nu}$. Then we have
\[
\mathbf{E}_{\mu,F} |\hat{\eta}_{\nu}(d-s_d) - \eta(\mu)| \leq s_d  F(- (a_d-t)) + (d-s_d)(1-F(t)).
\]
Since $a_d - t \geq (\nu \log(s_d) + \nu \log\log(d-s_d))^{1/\nu}$ and $F\in \mathcal{P}_\nu$ we get further that
\[
\mathbf{E}_{\mu,F} |\hat{\eta}_{\nu}(d-s_d) - \eta(\mu)| \leq \frac{2}{\log(d-s_d)} \underset{d \to \infty}{\to} 0.
\]
The proof of almost full recovery  property  \eqref{almost-full-rec-lc} for the selector $\widehat\eta=\hat{\eta}_{\nu}(d/s_d-1)$ follows along the same lines by setting $t=\sigma(\nu \log(d/s_d-1) + \nu \log\log(d/s_d-1))^{1/\nu}$. 
This yields
\[
\frac{1}{s_d}\mathbf{E}_{\mu,F} |\hat{\eta}_{\nu}(d/s_d-1) - \eta(\mu)| \leq \frac{2}{\log(d/s_d-1)} \underset{d \to \infty}{\to} 0.
\]

\noindent\textbf{Proof of Lemma \ref{lem:reduction}.}
We prove only \eqref{eq:lem:reduction1}. The proof of \eqref{eq:lem:reduction2} follows exactly the same argument, which in fact applies more generally, when the risk of selector $\widehat \eta$ is defined as ${\mathbf E}_\theta w(\widehat \eta, \eta(\theta))$ with a measurable function $w(\cdot,\cdot)$.

Let $O = (O_1,\dots,O_d)$ be a collection of random matrices such that $O_j$ are mutually independent, each $O_j$ is distributed according to the Haar measure $\Pi$ on the set of all orthogonal matrices in $\mathbf{R}^{k \times k}$, and $O$ is independent of $( Y_1,\dots,Y_d)$. For any tuple $o=(o_1,\dots,o_d)$ of orthogonal matrices $o_1,\dots,o_d$ we set
$$
\widetilde \eta_o ( Y_1,\dots,Y_d) = \hat \eta(o_1 Y_1,\dots, o_d Y_d)
$$
{By the invariance of the Haar measure $\Pi$ under orthogonal transformations the random matrices $O_j$ and ${O}_j  o_j$
have the same distribution for any orthogonal matrix $o_j$.
Choosing for $y_j \in \mathbf{R}^k$ an orthogonal matrix $o_j$ such that $o_j {\rm e}_1 \|y_j\| = y_j$ we have
\begin{equation}\label{samelaw}
O_j y_j \overset{D}{=} O_j {\rm e}_1 \|y_j\|, \quad j=1,\dots,d,
\end{equation}
where  ${\rm e}_1 = (1,0,\dots,0)$ is the first canonical basis vector of $\mathbf{R}^k$ and where $\overset{D}{=}\,$ means the equality in distribution of two random vectors. 
We conclude from \eqref{samelaw}, that the conditional distributions  of $O_j Y_j$ and  of $ O_j {\rm e}_1 \|Y_j\|$ given $Y_j$ coincide.} {Since $O_1,\dots,O_d$ are independent this conclusion extends to the joint distributions of $(O_1 Y_1,\dots,O_d Y_d)$  and $(O_1 {\rm e}_1 \|Y_1\|,\dots,O_d {\rm e}_1 \|Y_d\| )$ conditionally on $(Y_1,\dots,Y_d)$. This also implies that the unconditional joint distributions agree. }

{Consider now the randomized selector
$$
\overline \eta ( \|Y_1\|,\dots,\|Y_d\| ) := \hat \eta (O_1 {\rm e}_1 \|Y_1\|,\dots,O_d {\rm e}_1 \|Y_d\| ). 
$$
Then, as just proved,
$$
\overline \eta ( \|Y_1\|,\dots,\|Y_d\| ) \overset{D}{=} \hat \eta (O_1 Y_1,\dots,O_d Y_d)
 = \widetilde \eta_O(Y_1,\dots,Y_d).
$$
Thus
\begin{align}\label{eq:randomized0}
{\mathbf E}_\theta |\overline \eta - \eta (\theta)| = {\mathbf E}_\theta |\widetilde \eta_O - \eta(\theta)|=
\int {\mathbf E}_\theta |\widetilde \eta_o - \eta(\theta)| d \Pi(o) 
\end{align}
and
\begin{align}\label{eq:randomized}
\sup_{\theta \in \Theta_{k,d} (s,a)} \int {\mathbf E}_\theta |\widetilde \eta_o - \eta(\theta)| d \Pi(o)
& \leq  \sup_{\theta \in \Theta_{k,d} (s,a)} \sup_o {\mathbf E}_\theta |\widetilde \eta_o - \eta(\theta)| \\& =  \sup_o \sup_{\theta \in \Theta_{k,d} (s,a)} {\mathbf E}_\theta |\widetilde \eta_o - \eta(\theta)| ,\nonumber
\end{align}
where $\sup_o$ runs over all tuples $o=(o_1,\dots,o_d)$ of orthogonal matrices $o_1,\dots,o_d$. 
The result of the lemma now follows from  \eqref{eq:randomized0}, \eqref{eq:randomized}  and the equality
\begin{eqnarray}\label{samelaw2}
 \sup_o \sup_{\theta \in \Theta_{k,d} (s,a)}  {\mathbf E}_\theta |\widetilde \eta_o - \eta(\theta)| = \sup_{\theta \in \Theta_{k,d} (s,a)} {\mathbf E}_\theta |\hat \eta - \eta(\theta)|.
\end{eqnarray}
To prove \eqref{samelaw2}, we argue as follows.
For any fixed tuple $o=(o_1,\dots,o_d)$ of orthogonal matrices $o_1,\dots,o_d$ we have $\eta(\theta) = \eta(o\theta)$ with $o\theta= (o_1\theta_1,\dots,o_d\theta_d)$ and $o_j\xi_j\overset{D}{=} \xi_j$.  Therefore,
\begin{eqnarray*}
&& {\mathbf E}_\theta |\widetilde \eta_o - \eta(\theta)|  = {\mathbf E}_\theta |\hat \eta (o_1(\theta_1 + \sigma \xi_1),\dots,o_d(\theta_d + \sigma \xi_d)) - \eta(o \theta)| \\
&&=  {\mathbf E}_\theta |\hat \eta (o_1\theta_1 + \sigma \xi_1,\dots,o_d\theta_d + \sigma \xi_d) - \eta(o \theta)|\\
&& = {\mathbf E}_{o\theta} |\hat \eta (Y_1,\dots,Y_d)-  \eta(o \theta)|.
\end{eqnarray*}
It follows that
\begin{eqnarray*}
&& \sup_o \sup_{\theta \in \Theta_{k,d} (s,a)} {\mathbf E}_\theta |\widetilde \eta_o - \eta(\theta)|  = \sup_o \sup_{\theta \in \Theta_{k,d} (s,a)} {\mathbf E}_{o\theta} |\hat \eta (Y_1,\dots,Y_d)-  \eta(o \theta)|\\
&&=\sup_o \sup_{o\theta \in \Theta_{k,d} (s,a)} {\mathbf E}_{o\theta} |\hat \eta (Y_1,\dots,Y_d)- \eta(o \theta)|\\
&&= \sup_{\theta \in \Theta_{k,d} (s,a)} {\mathbf E}_{\theta} |\hat \eta (Y_1,\dots,Y_d)-  \eta( \theta)|,
\end{eqnarray*}
which implies the claim \eqref{samelaw2} and concludes the proof of the lemma.}

Below, we will use some probability inequalities for chi-square random variables. For convenience, we collect them in the following lemma.
\begin{lemma}\label{lem:chi2}
Let $x>0$. If $\chi$ is a $\chi^2$ random variable with $k$ degrees of freedom then
\begin{equation}\label{eq:chi2-LM}
\mathbf{P}(\chi \leq k-2 \sqrt{kx} ) \leq e^{-x}
\end{equation}
and there exist absolute constants $c_3>0$ and $c_4\in(0,1)$ such that
\begin{equation}\label{eq:chi2-ZZ}
\mathbf{P}( \chi \geq k+ c_3\max(x,\sqrt{kx}) ) \geq c_4 e^{-x}.
\end{equation}
If $Z$ is a $\chi^2$ random variable with $k$ degrees of freedom and non-centrality parameter $B\ge 0$ then
\begin{equation}\label{eq:chi2-Birge}
\mathbf{P}( Z \geq k+ B^2 + 2 \sqrt{(k+2B^2)x} + 2x) \leq e^{-x}.
\end{equation}
\end{lemma}
Inequalities \eqref{eq:chi2-LM}, \eqref{eq:chi2-ZZ}, and \eqref{eq:chi2-Birge} are proved in \cite[Lemma 1]{LM2000}, \cite[Corollary 3]{zhang2020nonasymptotic}, and \cite[Lemma 8.1]{birge}, respectively.

\noindent\textbf{Proof of Theorem \ref{th:lower-bound:group}.}
We assume w.l.o.g. that $\sigma=1$.
It suffices to prove \eqref{eq:exact-rec-lower} and \eqref{eq:almost-full-rec-lower} with $\Theta_{k,d} (s,a)$ replaced by a smaller set 
$$\Theta'_{k,d} (s,a)= \Big\{\theta\in\Theta_{k,d} (s,a) :\,\|\theta_j\|\in \{0,a\}, j=1,\dots,d\Big\}.
%\Theta_{k,d} (s,a)\cap \Big\{\theta:\,\|\theta_j\|\in \{0,a\}, j=1,\dots,d,\,\text{and}\, \sum_{j=1}^{d}\mathbf{1}(\|\theta_j\|=a)=s\Big\}.
$$
If $\theta$ belongs to $ \Theta'_{k,d} (s,a)$, the random variables 
$X_j = \|Y_j\|^2$ have a central $\chi^2$-distribution with $k$ degrees of freedom  for exactly $d-s$ indices $j$  and a non-central $\chi^2$-distribution with $k$ degrees of freedom  and with non-centrality parameter $a$ for the remaining $s$ indices $j$. Let $f_0$ and $f_1$ be the densities of these two distributions with respect to the Lebesgue measure:
\begin{eqnarray}\label{f0} f_0(z)  &=& \frac { z^{k/2-1} e^{-z/2}}{2^{k/2} \Gamma(k/2)},\\
f_1(z)  &=&2^{-k/2}  e^{-(z+a^2)/2} \sum_{j=0}^\infty \frac { z^{k/2+j-1} (a^2/4)^j}{\Gamma(k/2+j)j!},\quad z>0,\label{f1}
\end{eqnarray}
see \cite{JKB1995}. Notice that the likelihood ratio $f_1(z)/f_0(z)$ is monotone in $z$. Thus, we can use the lower bounds for variable selection under the MLR property obtained in Sections \ref{sec:minimax-and-bayes}. 

To prove the first assertion of the theorem we use Corollary \ref{cor:exact_rec}. It implies that if $F^{-1}_{0}(1-1/(d-s))> F^{-1}_{1}(1/s)$ then \eqref{eq:exact-rec-lower} holds. 
Thus,  it is enough to check that there exists $c_1>0$ such that, for all $d- s$ large enough, the condition  
$
a^{2} < c_1\max(\log(d-s), \sqrt{k\log(d-s)})
$
implies $F^{-1}_{0}(1-1/(d-s))> F^{-1}_{1}(1/s)$.
Note that, due to \eqref{eq:chi2-ZZ}, for all $d- s$ large enough we have
\[
F^{-1}_{0}(1-1/(d-s)) \geq k + c' \max(\log(d-s), \sqrt{k\log(d-s)})
\]
for some absolute constant $c'>0$. On the other hand,  \eqref{eq:chi2-Birge} with $B=a$ implies that, for all $s\ge2$, 
\begin{align*}
F^{-1}_{1}(1/s) &\leq k+ a^2 + 2 \sqrt{(k+2a^2)\log(s/(s-1))} + 2\log(s/(s-1))
\\
&\le 
k+ a^2 + 2 \sqrt{k+2a^2} + 2 \le k+ 2a^2 + 2 \sqrt{k} + 4.	
\end{align*}
Comparing the last two displays yields the result. 

The second assertion of the theorem is proved in a similar way.  We use Corollary \ref{cor:almost_rec}, which implies that  if $F^{-1}_{0}(1-1/(\lfloor d/s\rfloor-1))> F^{-1}_{1}(1/2)$ then \eqref{eq:almost-full-rec-lower} holds. By the same argument as above, $F^{-1}_1(1/2) \leq k+ 2a^2 + 2 \sqrt{k} + 4 $ and, for 
$d/s$ large enough,
\[
F^{-1}_{0}(1-1/(\lfloor d/s\rfloor -1)) \geq k + c'' (\log(d/s-1) \vee \sqrt{k\log(d/s-1)}),
\]
where $c''>0$ is an absolute constant. We conclude that there exists an absolute constant $c>0$ such that, for $d/s$ large enough, the condition
\[
a^{2} < c(\log((d-s)/s) \vee \sqrt{k\log((d-s)/s)})
\]
implies that $F^{-1}_{0}(1-1/(\lfloor d/s\rfloor-1))> F^{-1}_{1}(1/2)$. 

\noindent\textbf{Proof of Theorem \ref{th:upper-bound:group}.}
It is enough to prove the exact recovery and almost full recovery properties \eqref{exact-rec}  and \eqref{almost-full-rec} for the selector \eqref{opt-selector}. Indeed, then  the same properties 
for the scan selector follow from Proposition \ref{prop:relation-to-thresholding-lambda}, and the facts that   \eqref{opt-selector} is a thresholding selector and $X_1,\dots,X_d$ are $\mathbf{P}_\theta$-a.s. distinct for any $\theta\in\Theta_{k,d}(s,a)$. Moreover, 
\eqref{exact-rec-PWR} is an immediate consequence of \eqref{exact-rec} in view of  \eqref{eq:risks}.

We assume w.l.o.g. that $\sigma = 1$ and we write for brevity $s_d=s$, $k_d=k$, $a_d=a$, $t_d=t$. Denoting by $\langle \cdot,\cdot \rangle$ the Euclidean inner product we can bound the risk of the selector \eqref{opt-selector} as follows:
\begin{align*}
{\mathbf E}_\theta |\hat \eta - \eta(\theta)| &=  \sum_{i \in S(\theta) }\mathbf{P}_\theta(\|\theta_i + \xi_i\|^2 \leq t) + \sum_{i \not\in S(\theta)}\mathbf{P}_\theta(\|\xi_i\|^2 \geq t) \\
& \leq  \sum_{i \in S(\theta) }\mathbf{P}\left(\|\theta_i\|^2 - 2 \langle \theta_i,\xi_i \rangle +\| \xi_i\|^2  \leq t, \frac{\langle \theta_i,\xi_i \rangle}{\|\theta_i\|} < a/4\right) \\
& \quad + \sum_{i \in S(\theta) }\mathbf{P}\left(\frac{\langle \theta_i,\xi_i \rangle}{\|\theta_i\|} \geq a/4\right) + (d-s)\mathbf{P}( \|\xi_1\|^2 \geq t )
\\
& \leq  s \cdot \mathbf{P}( \|\xi_1\|^2 \leq t-a^{2}/2 ) %\\
%& \quad 
+ s \cdot \mathbf{P}\left(\mathcal{N} \geq a/4\right) + d \cdot \mathbf{P}( \|\xi_1\|^2 \geq t ),
%\\
%&\leq s\mathbf{P}( k - \|\xi_1\|^2 \geq a^{2}/2 -t^2 +  k) + s\mathbf{P}\left(\frac{\langle \theta_1,\xi \rangle}{\|\theta_1\|} \geq a/4\right)\\
%& \quad + d\mathbf{P}( \|\xi_1\|^2 -  k \geq t^2 - k).
\end{align*}
where $\mathcal{N}$ denotes a standard normal variable and we have used the fact that $\|\theta_i\|\ge a$ for $i \in S(\theta)$. We now apply \eqref{eq:chi2-LM} to bound $\mathbf{P}( \|\xi_1\|^2 \leq t-a^{2}/2 )$ and 
\eqref{eq:chi2-Birge} with $B=0$ to bound $\mathbf{P}( \|\xi_1\|^2  \geq t)$.
Since, by assumption, $t= k + 4 \log(d) + 4 \sqrt{k\log(d)}$ and $a^2/2 \geq t - k + 4\sqrt{k\log(d)} + 36 \log(d)$ we get that
\begin{align*}
{\mathbf E}_\theta |\hat \eta - \eta(\theta)| \leq s\Big(\frac{1}{d^4} + \frac{1}{d^2}\Big) + \frac1d \le \frac3d \underset{d \to \infty}{\to} 0. 
\end{align*}
Thus, the selector \eqref{opt-selector} achieves exact recovery  \eqref{exact-rec}  under the stated choice of $t$ and $a$.
The assertion of the theorem regarding almost full recovery is proved analogously. Indeed, the same argument as above with $t= k + 4 \log(d/s) + 4 \sqrt{k\log(d/s)}$ and $a^2/2 \geq t - k + 4\sqrt{k\log(d/s)} + 36 \log(d/s)$ yields that
\[
\frac{1}{s}{\mathbf E}_\theta |\hat \eta - \eta(\theta)| \leq \Big(\frac{1}{(d/s)^4} + \frac{1}{(d/s)^2}\Big) + \frac{s}{d}
\underset{d/s \to \infty}{\to} 0.
\]

\end{document}